\documentclass[10pt]{amsart}

\usepackage{amsmath,amsthm,amsfonts}
\usepackage{mathtools}
\usepackage{amssymb}
\usepackage{mathrsfs}
\usepackage{enumitem}
\usepackage{tikz-cd}
\usepackage{dsfont}
\usepackage[english]{babel} 
\usepackage{fullpage}
\usepackage{afterpage}
\usepackage{hyperref}
\usepackage[T1]{fontenc}
\usepackage[utf8]{inputenc}
\usepackage{lmodern}
\usepackage{graphicx} 
\usepackage{multicol} 
\usepackage{tikz} 
\usepackage{caption}
\usepackage{pgfplots} 
\usepackage{wrapfig}
\usepackage{array}
\usepackage{stackrel}
\usepackage{extarrows}
\usepackage{pgfplots}
\usetikzlibrary{decorations.markings}
\usepackage{mathtools}

\usepgfplotslibrary{polar}
\pgfplotsset{compat=newest}
\usetikzlibrary{calc}
\usetikzlibrary{graphs}
\usetikzlibrary{fit}
\usepgfplotslibrary{polar}
\pgfplotsset{compat=newest}
\usetikzlibrary{calc}

\newcommand{\Z}{\mathbb{Z}}
\newcommand{\G}{\Gamma}

\newcommand{\R}{\mathcal{R}}
\newcommand{\C}{\mathcal{C}}

\newcommand{\A}{{A_\Gamma}}

\newcommand{\Wh}{\mathrm{Wh}}

\newcommand{\Hom}{\mathrm{Hom}}

\newcommand{\FP}{\mathrm{FP}}
\newcommand{\Fsing}{\mathrm{sing}}

\newcommand{\MM}{\mathrm{MM}}

\newcommand{\lk}{\mathrm{lk}}
\newcommand{\st}{\mathrm{st}}
\newcommand{\Stab}{\mathrm{Stab}}
\newcommand{\Inn}{\mathrm{Inn}}
\newcommand{\PAut}{\mathrm{SPAut}}
\newcommand{\POut}{\mathrm{SPOut}}
\newcommand{\Aut}{\mathrm{Aut}}

\newcommand{\Who}{\mathrm{Wh}_\G^0}
\newcommand{\OO}{ \mathcal{O}}
\newcommand{\sil}{ \mathrm{sil}}
\newcommand{\boundellipse}[3]
{(#1) ellipse (#2 and #3)}
\newcommand{\abs}[1]{\left|#1\right|}

\newtheorem{teo}{Theorem}[section]
\newtheorem{lem}[teo]{Lemma}
\newtheorem{cor}[teo]{Corollary}
\newtheorem{prop}[teo]{Proposition}
\theoremstyle{remark}
\newtheorem{rem}[teo]{Remark}
\newtheorem{ej}[teo]{Example}
\theoremstyle{definition}
\newtheorem{defi}[teo]{Definition}

\theoremstyle{remark}
\newtheorem{ques}[teo]{Question}

\theoremstyle{plain}
\newtheorem{teoInt}{Theorem}

\theoremstyle{remark}
\newtheorem*{nota}{Remark}

\title{Actions with a strong fundamental domain and Sigma invariants of groups.}
\author{Peio Ardaiz Galé}
\author{Marcos Escartín Ferrer}
\author{Conchita Martínez Pérez}
\date{today}



\begin{document}

\thanks{\noindent
Partially supported by the Spanish Government PID2021-126254NB-I00 and PID2024-155800NB-C32. The first named author is also supported by Universidad P{\'u}blica de Navarra (grant code: Plan de promoci{\'o}n de grupos. ``{\'A}lgebra. Aplicaciones''). The second and third named authors are also supported by Departamento de Ciencia, Universidad y Sociedad del 
Conocimiento del Gobierno de Arag{\'o}n (grant code: E22-23R: ``{\'A}lgebra y Geometr{\'i}a'')}

\begin{abstract} We show that if a group $G$ acts on a contractible $CW$-complex with a contractible strong fundamental domain and the stabilizers of the action satisfy certain properties, then the homological Sigma invariants of the group are either empty or dense in the character sphere. This was known for several families of groups like right angled Artin groups and pure symmetric automorphisms of free groups. We also exhibit some applications including the generalization of the previous fact to pure symmetric automorphisms of right angled Artin groups.
\end{abstract}

\maketitle

\section{Introduction}

One of the main ideas in geometric group theory is to understand groups via their actions. This is particularly useful when one is interested in cohomological finiteness properties, such as being of type $\mathrm{F}_n$ or $\mathrm{FP}_n$. In this setting, there are many results where a group $G$ acts nicely on a topological space, typically a contractible $CW$-complex, and, assuming that the cell stabilizers satisfy certain properties, it is possible to derive similar properties for the group $G$. See, for example, \cite{Brown}, where Brown proves some paradigmatic results along these lines.

This strategy was also used by Meier, Meinert, and VanWyk in \cite{MMVW} to compute the Sigma invariants  of right-angled Artin groups (RAAGs). The Sigma invariants of a group $G$ (also called BNSR invariants, after Bieri, Neumann, Renz, and Strebel; see Section \ref{sec:Sigma} for definitions) form an important family of geometric invariants that provide a complete description of the algebraic fibrations of $G$. Recall that a group $G$ algebraically fibers if there exists a homomorphism $G \to \mathbb{Z}$ with finitely generated kernel. If the kernel satisfies stronger finiteness properties, such as being of type $\mathrm{F}_n$ or $\mathrm{FP}_n$, one speaks of higher-dimensional fibrations. Sigma invariants consist of two families, one of homological and one of homotopical nature, the first is related to fibrations where the kernel is of type $\mathrm{FP}_n$, the second to fibrations where the kernel is of type $\mathrm{F}_n$. The notion of algebraic fibration was used by Kielak in \cite{Kielak} to extend to a purely group-theoretic setting previous breakthrough work by Agol in 3-manifold theory, in his work Kielak also showed that fibrations are connected to $\ell^{(2)}$-Betti numbers: if the group $G$ has an $n$-dimensional virtual homological  fibration, then $\beta^{(2)}_n(G)=0$; the converse holds if $G$ is virtually residually finite rationally solvable (cf. \cite{Kielak}, \cite{Sam}).  

Recently, Ershov and Zaremsky \cite{ErshovZaremsky} considered the group $G=\mathrm{POut}(F_n)$ of pure symmetric outer automorphisms  of a free group $F_n$ (recall that pure symmetric automorphisms are those sending each basis element to a conjugate of itself) and use
the action of $G$ on certain space to prove that for any $m\geq 0$ there is a dichotomy: the invariant $\Sigma^m$ is either  empty or dense.  This has very strong implications for the possible fibrations of the group. For example, in the case when the homological $\Sigma^m$ is empty, this means that there is no homomorphism $G \to \mathbb{Z}$ whose kernel is of type $\mathrm{FP}_m$; if the homological $\Sigma^m$ is dense, it means that a small perturbation of any homomorphism $G \to \mathbb{Z}$ has kernel of type $\mathrm{FP}_m$. Analogous statements hold for the homotopical version. In the same paper, Ershov and Zaremsky are able to characterize the values of $n$ and $m$ for which each of these possibilities happens. The space used in \cite{ErshovZaremsky} is the McCullough--Miller space constructed in \cite{MM} in a more general setting. In the particular case of the groups $\mathrm{POut}(F_n)$, one obtains an action  with free abelian stabilizers. 

This dichotomy is a common phenomenon, because in many of the cases where a complete description of the $\Sigma$ invariants of a group is available, it turns out that they are either empty or the result of removing finitely many linear subspaces from the sphere $S(G)$.  For example, Koban and Piggot \cite{KoPi} showed that this is the case for $\Sigma^1(\PAut(\A))$  where $\PAut(\A)$ is the group of pure symmetric automorphisms of the RAAG $A_\Gamma$ and Day and Wade \cite{DayWade} did the same for the external version $\POut(\A)$. However, there are also known examples for which this is not true, as Thompson group $F$ (cf. \cite[Theorem A]{BK})

In this paper, motivated by a question of Zaremsky, we show that the same dichotomy holds for the higher homological Sigma invariants of $\POut(\A)$.  The characterization of the values of $m$ for which the invariant $\Sigma^m$ is dense is given in terms of a poset called the  $\Gamma$-\emph{Whitehead} poset (see Section \ref{sec:MM}), the geometric realization of this poset is the  fundamental domain for the action of the group on a generalization of the McCullough-Miller space constructed previously by the first and last authors in \cite{ArMar}. 

To do this, we first prove a more general theorem that can be applied to other relevant families of groups. The objective is to give conditions, in terms of an action on a suitable space $X$, that imply that for a given group $G$ the (homological) Sigma invariants are either dense or empty, and also to show that using the fundamental domain of the action one can determine which of the two cases occurs. 

We assume the following general hypothesis:

\bigskip

\noindent{\bf Hypothesis (*)}: $G$ is a group acting cocompactly, admissibly, and with stabilizers of type $\mathrm{FP}_\infty$ on a contractible $CW$-complex $X$, with a contractible sub $CW$-complex $W\subseteq X$ which is a  strong fundamental domain.

\bigskip

Given a group action $G\curvearrowright X$ as in Hypothesis (*), we consider the following subcomplex of $W$:
$$W_{\Fsing}=\{\sigma\in W \mid G_\sigma \neq 1\}$$
where $G_\sigma$ is the stabilizer of the cell $\sigma$. We also denote
\[\Omega_\sigma=\{[\chi]\in S(G)\mid[\chi|_{G_\sigma}]\in\Sigma^n(G_\sigma,\Z)\}\]
and show:

\begin{teoInt}\label{teo:A}
Let $G\curvearrowright X$ be as in Hypothesis (*), and assume that for each cell $\sigma\in W_\Fsing$, $\Omega_\sigma\subseteq S(G)$ is dense.
Then the following are equivalent:
\begin{itemize}
  \item[(i)] $\Sigma^n(G,\mathbb{Z})$ is dense in $S(G)$,
    \item[(ii)] $\Sigma^n(G,\mathbb{Z})\cap -\Sigma^n(G,\mathbb{Z})$ is dense in $S(G)$,
    \item[(iii)] $\Sigma^n(G,\mathbb{Z})\cap -\Sigma^n(G,\mathbb{Z})\neq\emptyset$,
    \item[(iv)] $G$  fibers homologically at degree $n$,
    \item[(v)] $W_{\Fsing}$ is $(n-1)$-acyclic.
\end{itemize}
\end{teoInt}

The hypothesis that the set $\Omega_\sigma\subseteq S(G)$ is dense can be difficult to check in practice, in Lemma \ref{lem:densitycond} we have listed a series of easier conditions that imply it.

We also get, see Theorem \ref{teo:homotopicA}, a weaker version of Theorem \ref{teo:A} for the homotopical Sigma invariants: assuming that $W_\Fsing$ is simply connected and that for any $\sigma\in W_\Fsing$, the set 
\[\Omega^{t}_\sigma=\{[\chi]\in S(G)\mid[\chi|_{G_\sigma}]\in\Sigma^n(G_\sigma)\}\]
is dense in $S(G)$, then the following are equivalent:
\begin{itemize}
     \item[(i)] $\Sigma^n(G)$ is dense in $S(G)$,
       \item[(ii)] $\Sigma^n(G)\cap -\Sigma^n(G)$ is dense in $S(G)$,
    \item[(iii)] $\Sigma^n(G)\cap -\Sigma^n(G)\neq\emptyset$,
    \item[(iv)] $G$  fibers homotopically at degree $n$,
    \item[(v)] $W_{\Fsing}$ is $(n-1)$-connected.
\end{itemize}

Sometimes, the Sigma invariants are symmetric, i.e. $\Sigma^n(G,\Z)=-\Sigma^n(G,\Z)$ and $\Sigma^n(G)=-\Sigma^n(G)$, this is the case when the group has a generating system $X$ such that $f:X\to G$, $f(a)=a^{-1}$ extends to a well defined group isomorphism \cite[Definition 2.4]{ErshovZaremsky}. For those groups, item (iii) can be shorten to $\Sigma^n(G,\mathbb{Z})\neq\emptyset$, or the analogous homotopical formulation.

Although our original motivation was to apply Theorem \ref{teo:A} to pure symmetric automorphisms of RAAGs, it turns out that it can be used in many other cases. We show:


\begin{teoInt}\label{teo:families}
    For the following families of groups $G$ there is an action that satisfies the hypotheses of Theorem \ref{teo:A}:
\begin{itemize}
 \item[(i)] Fouxé-Rabinovitch automorphisms groups of free products of $\mathrm{FP}_\infty$ such that the factors have dense Sigma invariants. In this case, if $n$ is the number of factors, $\Sigma^k(G,\Z)$ is dense in $S(G)$ if and only if $k< n-2$ (cf. Corollary \ref{cor:FR}),
 
\item[(ii)] groups of symmetric pure automorphisms of right-angled Artin groups (cf.  Theorem \ref{teo:firstPartPOutRAAGs}),

    \item[(iii)] certain fundamental groups of simple complexes of groups (cf.  Theorem \ref{theorem:SimpleComplexGroups}),

     \item[(iv)] Artin groups satisfying the $K(\pi,1)$-conjecture (cf.  Theorem \ref{theorem:SigmaArtin}),
    
    \item[(v)] graph products of $\mathrm{FP}_\infty$ groups that have dense Sigma invariants (cf.  Theorem \ref{theorem:SigmaGraphProd}).
\end{itemize}
\end{teoInt}

For  these families, the fundamental domain of the action depends on the construction and can be more or less explicitly described.  
In the case of the groups $\POut(A_\G)$ where $A_\G$ is a RAAG, the singular subcomplex of the fundamental domain is the $\G$-Whitehead poset $|\Who|$ that was constructed in \cite{ArMar} and we describe in Section \ref{sec:MM}. This complex can be rather complicated, however in some cases, as for example if the defining graph $\G$ is a tree, the connectivity of its geometric realization is much easier to describe. We show

\begin{teoInt}\label{teo:POutRAAGs}
    Let $G=\POut(\A)$ be the pure symmetric outer-automorphism group of a \textup{RAAG} $\A$. Then, for any $k$, $\Sigma^k(G,\Z)$ is  dense in $S(G)$ if and only if  $|\Who|$ is $(k-1)$-acyclic and in other case it is empty. Moreover
    \begin{itemize}
         \item[(i)] If the support graph $\Delta_v$ has two or more connected components for any $v\in\G$, then $\Sigma^{\infty}(G,\Z)$ is dense in $S(G)$.
        \item[(ii)] If $\G$ is a tree, the only cases when $\Sigma^k(G,\mathbb{Z})$ is empty are:
    \begin{itemize}
        \item[$\bullet$] $\G$ is a star with $n$ leaves and  $k\geq n-2$, or
        \item[$\bullet$] $\G$ consists of two stars $S_1,S_2$ with $n_1$ and $n_2$ leaves respectively, linked by an edge between the two internal nodes and $k\geq n_1+n_2-2$.
    \end{itemize}
        \end{itemize}
\end{teoInt}

We treat the case of $\Sigma^1$ separately, because the fact that for groups of pure symmetric automorphisms of RAAGs $\Sigma^1$ is either dense or empty was previously known \cite{DayWade}. Using results of the second named author in \cite{Escartin} one can also characterize when $\Sigma^1$ is dense and show that for these groups this is equivalent to the vanishing of the first  $\ell^{(2)}$-Betti number. In Theorem \ref{teo:sigma1} we recall this characterization and add a couple of equivalent conditions in terms of the poset $\Who$.

In \cite{DayWade}, Day and Wade use Sigma invariants to  characterize when the group $\POut(A_\G)$ is isomorphic to a RAAG. There is a well-known action for RAAGs $A_\Theta$ that satisfies Hypothesis (*), namely the action on the Deligne complex, for which the singular subcomplex of the fundamental domain is the flag complex $\widehat{\Theta}$ associated to the defining graph.
This means that in this case we can apply Theorem \ref{teo:A} either for the complex $|\Wh_\G^0|$ or else for $\widehat{\Theta}$ and as a by-product we deduce that for any $n$, one is $n$-acyclic if and only if the other is. So a natural question is whether they are homotopy equivalent and we show that this is indeed the case.

\begin{teoInt}\label{teo:homotopWhi-RAAG}
    If $\POut(\A)$ is a \textup{RAAG} $A_\Theta$, then there is a homotopy equivalence $|\Wh_\G^0|\simeq \widehat{\Theta}$. 
\end{teoInt}

Wiedmer has shown, cf.  \cite[Theorem 3.1]{Wied}, that given any graph $\Theta$ there is another graph $\G$ such that  $\POut(A_{\G})\cong A_\Theta$. This together with Theorem \ref{teo:homotopWhi-RAAG} imply that $|\Wh_\G^0|$ can have any  homotopy/homology type.

{\bf Structure of the paper.} In Section \ref{sec:Sigma} we introduce Sigma invariants and recall some of their properties. Section \ref{sec:Mainresult} is where Theorem \ref{teo:A} is proven, we also give in Lemma \ref{lem:densitycond} a list of conditions which are useful to check the hypotheses of Theorem \ref{teo:A}. The rest of the paper is devoted to applications. In Section \ref{sec:Fouxe} we consider Fouxé-Rabinovitch groups and show item (i) of Theorem \ref{teo:families} and in Sections \ref{sec:MM} and \ref{sec:supp} we deal with groups of pure symmetric automorphisms of right-angled Artin groups: in Section \ref{sec:MM} we recall the construction of the $\G$-McCullough Miller complex and show that Theorem \ref{teo:A} applies, which yields item (ii) of Theorem \ref{teo:families}, and in Section \ref{sec:supp} we show Theorems \ref{teo:POutRAAGs} and \ref{teo:homotopWhi-RAAG}. Finally, in Section \ref{sec:simplecomplexes} we observe that under mild conditions, the natural action of fundamental groups of certain simple complexes of groups satisfies Hypotesis (*)  and show Theorems \ref{theorem:SigmaArtin} (for Artin groups satisfying the $K(\pi,1)$-conjecture) and  \ref{theorem:SigmaGraphProd} (for graph products), these results prove items (iii), (iv) and (v) of Theorem \ref{teo:families}.

\section{Sigma invariants and some previous results}\label{sec:Sigma}
Let $G$ be a finitely generated group. Define the \emph{character sphere} of $G$ as
$$
S(G)=(\mathrm{Hom}(G,\mathbb{R})\setminus\{0\})/\sim
$$
where two non-zero characters $\chi_1,\chi_2:G\to\mathbb{R}$ are equivalent if $\chi_1=t\chi_2$ for some $t>0$. If $n$ denotes the torsion-free rank of the abelianization of $G$, then $S(G)\cong\mathbb{S}^{n-1}$ (c.f. \cite[Lemma A1.1]{Strebel}). A character $\chi:G\to \mathbb{Z}$ is called \emph{discrete}, and the set of all discrete characters is dense in $S(G)$ (c.f. \cite[Lemma B3.24]{Strebel}). We define the homological and homotopical Sigma invariants of a group as follows:

\begin{defi}
    Let $G$ be a finitely generated group and $\chi:G\to\mathbb{R}$ a character. If $A$ is a left $G$-module and $n\geq 1$, the \emph{Bieri–Neumann–Strebel–Renz (BNSR) invariants} are defined by
    $$
    \Sigma^n(G,A)=\{[\chi]\in S(G)\mid A \text{ is of type }\mathrm{FP}_n \text{ over }\mathbb{Z}G_\chi\}
    $$
 where $G_\chi=\{g\in G\mid\chi(g)\geq 0\}$.   Observe that
    $$
    \Sigma^\infty(G,A)\subseteq \dots \subseteq \Sigma^n(G,A)\subseteq \dots \subseteq \Sigma^1(G,A)\subseteq \Sigma^0(G,A)=S(G).
    $$
\end{defi}

\begin{defi}
    Let $G$ be a group admitting a finite $K(G,1)$, and let $X$ be a finite model for $K(G,1)$ having a single 0-cell. There is a bijection between the vertices of the universal cover $\widetilde{X}$ and the elements of $G$. Fix a base point $b\in\widetilde{X}^{(0)}$. Given a character $\chi:G\to\mathbb{R}$, define a function $\widetilde{\chi}:\widetilde{X}^{(0)}\to\mathbb{R}$ by $\widetilde{\chi}(g\cdot b)=\chi(g)$. For $a\in\mathbb{R}$, let $\widetilde{X_{\chi}}^{[a,\infty)}$ be the maximal subcomplex of $\widetilde{X}$ with $0$-skeleton $\{x\in\widetilde{X}^{(0)}\mid \widetilde{\chi}(x)\geq a\}$. Then, for each $d\geq 0$, the inclusion $\widetilde{X_{\chi}}^{[0,\infty)}\hookrightarrow \widetilde{X_{\chi}}^{[-d,\infty)}$ induces a map $\tau_i^d:\pi_i\bigl(\widetilde{X_{\chi}}^{[0,\infty)}\bigr)\longrightarrow \pi_i\bigl(\widetilde{X_{\chi}}^{[-d,\infty)}\bigr).$
    Define the \emph{homotopical Sigma invariants} by
    $$\Sigma^n(G)=\Bigl\{[\chi]\in S(G)\ \Bigm|\ \exists\, d>0 \text{ such that }\tau_i^d\text{ is trivial for all } i<n \Bigr\}.$$
    This yields a descending filtration:
    $$\Sigma^\infty(G)\subseteq \dots \subseteq \Sigma^1(G)\subseteq \Sigma^0(G)=S(G). $$
\end{defi}
Renz proved in his thesis (cf. \cite{Renz}) that the homotopical Sigma invariants are well-defined, i.e. they do not depend on the chosen model for $K(G,1)$. Bieri and Renz also showed that both $\Sigma^n(G,\mathbb{Z})$ and $\Sigma^n(G)$ are open in $S(G)$, and that if $n=\mathrm{cd}(G)$ is the cohomological dimension of $G$, then $\Sigma^m(G,\mathbb{Z})=\Sigma^n(G,\mathbb{Z})$ for every $m\geq n$. Similarly, if $n=\mathrm{gd}(G)$ is the geometric dimension of $G$, then $\Sigma^m(G)=\Sigma^n(G)$ for every $m\geq n$ (c.f.\ \cite{Bieri-Renz}).

Another important remark is that the homotopical and homological Sigma invariants are closely related. Indeed, $\Sigma^1(G)=\Sigma^1(G,\mathbb{Z})$ and $\Sigma^n(G)=\Sigma^n(G,\mathbb{Z})\cap\Sigma^2(G)$ for every $n\geq 1$ (c.f.\ \cite{Renz}). The main result about Sigma invariants, which explains why they encode the finiteness properties of the coabelian subgroups, is the following:
\begin{teo}[Bieri-Renz \cite{Bieri-Renz}]\label{teo:FinitenesSigma}
	Let $G$ be a group of type $\mathrm{F}_n$ and $G'\leq N\leq G$. Then:
	\begin{enumerate}
		\item[(i)] $N$ is of type $\mathrm{FP}_n$ if and only if $\{[\chi]\in S(G)\mid \chi(N)=0\}\subseteq \Sigma^n(G,\mathbb{Z})$.
		\item[(ii)] $N$ is of type $\mathrm{F}_n$ if and only if
		$\{[\chi]\in S(G)\mid \chi(N)=0\}\subseteq \Sigma^n(G)$.
	\end{enumerate}
\end{teo}
Now we introduce two results that will be relevant later.
\begin{lem}\cite[Lemma 2.1]{MeierMeinert}\label{lem:SigmaCentre}
	Let $G$ be a group of type $\mathrm{F}_n$ and let $\chi:G\to\mathbb{R}$ be a character with $\chi(\mathrm{Z}(G))\neq 0$ where $Z(G)$ is the center of $G$. Then $[\chi]\in \Sigma^n(G)\subseteq \Sigma^n(G,\mathbb{Z})$.
\end{lem}

\begin{rem}\label{rem:centerdense} As a consequence, if $G$ is a group of type $F_\infty$ such that the image of $Z(G)$ under the abelianization map $G\to G_{\mathrm{ab}}=G/G'$ is infinite, then $\Sigma^\infty(G)$ is dense in $S(G)$ because if we set
\[\mathrm{Ann}(Z(G))=\{[\chi]\in S(G)\mid\chi(Z(G))=0\},\]
then $\mathrm{Ann}(Z(G))$ is the intersection with $S(G)$ of a proper linear subspace of $\mathrm{Hom}(G,\mathbb{R})$ so $S(G)-\mathrm{Ann}(Z(G))$ is dense in $S(G)$ and Lemma \ref{lem:SigmaCentre} implies
\[S(G)-\mathrm{Ann}(Z(G))\subseteq\Sigma^\infty(G).\]    
\end{rem}

The following Theorem is due to Schmitt (homological version) and Meinert (homotopical version).

\begin{teo}\cite[Theorem 3.2]{MMVW}\label{teo:MMVW}
Let $X$ be a $G$–CW–complex with cocompact $n$-skeleton, and let $\chi:G\to\mathbb{R}$ be a character such that for every $p$-cell $\sigma$ of $X$ with $p\leq n$, we have $\chi(G_\sigma)\neq 0$, where $G_\sigma$ denotes the stabilizer of $\sigma$. Then:
\begin{itemize}
\item[(i)] If $X$ is $(n-1)$-connected and $[\chi|_{G_\sigma}]\in\Sigma^{\,n-p}(G_\sigma)$ for each $p$-cell $\sigma$ with $p\leq n$, then $[\chi]\in\Sigma^n(G)$.

\item[(ii)] If $X$ is $(n-1)$-acyclic and $[\chi|_{G_\sigma}]\in\Sigma^{\,n-p}(G_\sigma,\mathbb{Z})$ for each $p$-cell $\sigma$ with $p\leq n$, then $[\chi]\in\Sigma^n(G,\mathbb{Z})$.
\end{itemize}
\end{teo}




\section{Actions with a nice strong fundamental domain and Sigma invariants}\label{sec:Mainresult}

In this section we prove Theorem \ref{teo:A}. Throughout the section, we assume that we have a group action $G\curvearrowright X$ satisfying Hypothesis (*) from the introduction, which we recall here.

\bigskip

\noindent{\bf Hypothesis (*)}: $G$ is a group acting cocompactly, admissibly, and with stabilizers of type $\mathrm{FP}_\infty$ on a contractible $CW$-complex $X$, with a sub $CW$-complex $W\subseteq X$ which is a contractible strong fundamental domain.

\bigskip

Recall that $W$ is a strong fundamental domain if any $G$-orbit of cells in $X$ has precisely one representative in $W$. Observe also that the cocompactness hypothesis implies that $W$ is finite, and together with the assumption that the cell stabilizers are of type $\mathrm{FP}_\infty$, it follows that $G$ is also of type $\mathrm{FP}_\infty$ by Brown's result mentioned in the introduction \cite[Proposition 1.1]{Brown}. 

Recall also that $W_{\Fsing}$ denotes the subcomplex of $W$ consisting of those cells $\sigma$ such that the stabilizer $G_\sigma$ is nontrivial. If $\delta\subseteq\sigma$ is a subcell, then the admissibility hypothesis implies that $G_\sigma \leq G_\delta$, so $W_{\Fsing}$ is indeed a subcomplex.

\begin{prop}\label{prop:homcond} Let $G\curvearrowright X$ be a group action as in Hypothesis (*) with strong fundamental domain $W$ and let $0\neq\chi:G\to\Z$ be a discrete character with kernel $N=\ker\chi$. Assume that for each cell $\sigma\in W_\Fsing$, $\chi(G_\sigma)\neq 0$ and that the homology groups $H_i(N\cap G_\sigma,\Z)$ are finitely generated (as abelian groups) for $i=0,\ldots,n$. If  $W_{\Fsing}$ is not $(n-1)$-acyclic, then there is some $0<i\leq n$ such that the homology group $H_i(N,\Z)$ is not finitely generated (as abelian group).
\end{prop}
\begin{proof} Assume, for a contradiction, that $H_i(N,\Z)$ is finitely generated as abelian group for each $0\leq i\leq n$.
As $W_{\Fsing}$ is not $(n-1)$-acyclic, there is some $0\leq j\leq n-1$ such that $\overline{H}_j(W_{\Fsing})\neq 0$ where $\overline{H}$ denotes reduced homology. As $W$ is a finite complex,  $\overline{H}_j(W_{\Fsing})\neq 0$ is finitely generated so we may choose $K$ either the rational field or a prime field of the form $\Z_p$ such that $\overline{H}_j(W_{\Fsing})\otimes_\Z K\neq 0$. The universal coefficients theorem implies that then
\begin{equation}\label{eq:Wsing}
    \overline{H}_j(W_{\Fsing},K)\neq 0.
\end{equation}
Using the universal coefficients theorem again
$$H_i(N,K)=H_i(N,\Z)\otimes_\Z K\oplus\mathrm{Tor}_1^\Z(H_{i-1}(N,\Z),K)$$
and therefore $H_i(N,K)$ has finite $K$-dimension for each $0\leq i\leq n$.

 Now, let $R=K[G/N]$ be the group ring with coefficients in $K$ of the cyclic group $G/N$. As $G$-module, $R$ is the induced module $R=K\uparrow_N^G$. On the other hand, if we denote by $t$ a generator of $G/N$, then $R=K[t^{\pm 1}]$ is a principal domain. By Shapiro's lemma
 $$H_i(N,K)=H_i(G,K\uparrow_N^G)=H_i(G,R)$$ 
 and as $H_i(G,R)$ is a finitely generated $R$-module, it decomposes as the direct sum of an $R$-free part and an $R$-torsion submodule, but being of finite $K$-dimension, the $R$-free part must be trivial. Let $F$ be the field of fractions of $R$. Applying the functor $-\otimes_RF$ kills the $R$-torsion part, in other words, for $0\leq i\leq n$ we have
 $$0=H_i(G,R)\otimes_RF=H_i(G,F)$$
 where the last equality follows from the fact that $F$ is flat as $R$-module and $H_i(G,R)$ is the homology of the complex $P_\bullet\otimes_GR$ where $P_\bullet$ is a $\Z G$-projective resolution of the trivial module $\Z$.

 By \cite[Chapter VII (7.10)]{BrownBook}, the homology groups $H_\bullet(G,F)$ can be computed using the complex $X$ and a
 spectral sequence 
 $$E^1_{p,q}\Rightarrow H_{p+q}(G,F)$$
  having first page 
  $$E^1_{p,q}=\bigoplus_{\sigma \text{$p$-cell in }W}H_{q}(G_\sigma,F).$$
  Here we are implicitly using the fact that the action of $G$ on $X$ is admissible so cell stabilizers $G_\sigma$ fix the cells $\sigma$ pointwise and in particular preserve the orientation (in Brown's formulation one has to twist the coefficients module by the orientation character of $G_\sigma$ acting on $\sigma$). The differential  $d^1:E^1_{p,q}\to E^1_{p-1,q}$ comes from the boundary map in $X$.
 
 Let $\sigma\in W_\Fsing$ be a $p$-cell. Our hypothesis imply that  $\chi(G_\sigma)\neq 0$ so $G_\sigma\not\leq N$ which taking into account that $G/N$ is cyclic implies that the index $|G:NG_\sigma|$ is finite. Therefore
 the hypothesis also imply that for $0\leq q\leq n$  
 $$H_q(G_\sigma,R)=H_q(G_\sigma,K\uparrow_N^G\downarrow^G_{G_\sigma})=\bigoplus_{|G:NG_\sigma|}H_q(G_\sigma,K\uparrow_{N\cap G_\sigma}^{G_\sigma})=\bigoplus_{|G:NG_\sigma|}H_q(N\cap G_\sigma,K)$$
  has finite $K$-dimension. So the same argument as before yields 
 $$H_{q}(G_\sigma,F)=0.$$
 If $\sigma$ is a $p$-cell such that $G_\sigma=1$, then 
  $$H_{q}(G_\sigma,F)=\Bigg\{
  \begin{aligned}
 &F \text{ if }q=0\\
  &0\text{ otherwise}.\\
 \end{aligned} $$
 Therefore, $E^1_{p,q}=0$ if $q>0$ and the first page of the spectral sequence is a single row with
 $$E^1_{p,0}=\bigoplus_{\sigma \text{$p$-cell in $W$ with $G_\sigma=1$}}F$$
 and differential coming from the differential in $X$. This is precisely the $p$-th term $\mathcal{C}_p(W,W_{\Fsing})$ of the chain complex of the pair $(W,W_{\Fsing})$ after $\Z$-tensoring with the field $F$. Therefore
 $$H_p(W,W_{\Fsing},K)\otimes_K F=H_p(\mathcal{C}_\bullet(W,W_{\Fsing})\otimes_K F)=H_p(G,F)$$
where $H_p(W, W_{\Fsing}, K)$ denotes ordinary relative homology with coefficients in $K$, and since $F$ is $K$-flat, tensoring with $F$ preserves homology.  Now, using the long exact reduced homology sequence associated  to the pair $(W,W_{\Fsing})$ with coefficients in $K$ we get
 $$\ldots\to\overline{H}_{p}(W,K)\to H_p(W,W_{\Fsing},K)\to \overline{H}_{p-1}(W_{\Fsing},K)\to\overline{H}_{p-1}(W,K)\to\ldots$$
 and since we are assuming that $W$ is contractible we get for $p\geq 1$
 $$ \overline{H}_{p-1}(W_{\Fsing},K)\otimes_K F=H_p(W,W_{\Fsing},K)\otimes_K F=H_p(G,F).$$
 Finally, let $0\leq j\leq n-1$ be the  index of (\ref{eq:Wsing}). As $0\neq\overline{H}_j(W_{\Fsing},K)$ is a $K$-vector space and $K\subseteq F$ a field extension we have
 $$0\neq\overline{H}_j(W_{\Fsing},K)\otimes_KF=H_{j+1}(G,F)$$
 which is a contradiction.
\end{proof}
    
   \begin{teo}\label{teo:charsigma}
Let $G$, $X$, $W$ be as in Hypothesis (*) and let $0\neq \chi:G\to\Z$ be a discrete character. Assume that for each cell $\sigma \in W_{\Fsing}$, $\pm[\chi] \in \Sigma^n(G_\sigma,\Z)$. Then $\pm[\chi]\in\Sigma^n(G,\Z)$ if and only if $W_{\Fsing}$ is $(n-1)$-acyclic.
\end{teo}

\begin{proof}

Assume that $W_{\Fsing}$ is not $(n-1)$-acyclic. 
By hypothesis, for each cell $\sigma\in W_\Fsing$, we have $\pm[\chi] \in \Sigma^n(G_\sigma, \Z)$, which implies that $\ker\chi \cap G_\sigma$ is of type $\FP_n$. Therefore, the homology groups $H_i(\ker\chi \cap G_\sigma, \Z)$ 
are finitely generated for $i = 0, \ldots, n$. 

Applying Proposition \ref{prop:homcond}, we deduce that there exists some $0 \le j \le n$ such that $H_j(\ker\chi, \Z)$ is not finitely generated. Hence, $\ker\chi$ is not of type $\FP_n$, and we conclude that $\pm[\chi] \notin \Sigma^n(G, \Z)$.

Assume now that $W_{\Fsing}$ is $(n-1)$-acyclic. The long exact sequence of homology for the pair $(W, W_{\Fsing})$ gives $H_i(W, W_{\Fsing}) = 0 \quad \text{for } 1 \le i \le n,$ because $W$ is contractible. 

Let $X_{\Fsing} = \{\sigma \subset X \mid G_\sigma \neq 1\}$ be the subcomplex of $X$ consisting of cells with non-trivial stabilizers. Then, each open cell in $X - X_{\Fsing}$ lies in a free $G$-orbit, and since $W$ is a strong fundamental domain, each orbit has precisely one representative in $W - W_{\Fsing}$. In other words, as sets of cells, we have
$$X - X_{\Fsing} = G \times (W - W_{\Fsing}).$$

It follows that the chain complex of the pair $(X, X_{\Fsing})$ can be written as
$$\mathcal{C}_\bullet(X, X_{\Fsing}) = \Z G \otimes_\Z \mathcal{C}_\bullet(W, W_{\Fsing}),$$
and therefore the homology satisfies
$$H_i(X, X_{\Fsing}) = \Z G \otimes_\Z H_i(W, W_{\Fsing}) = 0 \quad \text{for } 1 \le i \le n.$$

Using the long exact homology sequence for the pair $(X, X_{\Fsing})$ and the fact that $X$ is contractible, we obtain $\overline{H}_{i-1}(X_{\Fsing}) = 0 \quad \text{for } 1 \le i \le n,$
so the subcomplex $Y := X_{\Fsing}$ is $(n-1)$-acyclic. 

At this point, we have an $(n-1)$-acyclic $G$-space $Y$ with cocompact $G$-action, and a character $\chi$ such that for each cell stabilizer $G_\sigma$ of $Y$, the restrictions $\pm[\chi|_{G_\sigma}]$ are non-zero and lie in $\Sigma^n(G_\sigma, \Z)$. 

Applying Theorem \ref{teo:MMVW}, we conclude that $\pm[\chi] \in \Sigma^n(G, \Z)$.
\end{proof}
    
   \begin{rem}\label{rem:homotopic} {\rm 
The homotopical version of the ``if'' part also holds true, i.e., under the same hypotheses, if $W_{\Fsing}$ is $(n-1)$-connected, then $\pm[\chi]\in\Sigma^n(G)$. To see this, note that one can replace the long exact sequences in homology with the corresponding long exact sequences in homotopy to deduce that if $W_{\Fsing}$ is $(n-1)$-connected, then $X_{\Fsing}$ is also $(n-1)$-connected. The result then follows by applying the homotopical version of Theorem \ref{teo:MMVW}.
}\end{rem}


Let $\sigma\in W_\Fsing$ be a cell and put
\[\Omega_\sigma=\{[\chi]\in S(G)\mid[\chi|_{G_\sigma}]\in\Sigma^n(G_\sigma,\Z)\}.\]
Then 

\begin{lem}\label{lem:density}
    Let $G,X,W$ as in Hypothesis (*) and assume that $\Omega_\sigma\subseteq S(G)$ is dense for any $\sigma\in W_\Fsing$. Then
    \[\Omega_\Z:=\{[\chi]\in S(G)\ \text{discrete}\mid [\chi|_{G_\sigma}]\in\Sigma^n(G_\sigma,\Z)\cap-\Sigma^n(G_\sigma,\Z)\ \text{for every cell }\sigma\in W_{\Fsing}\}\]
    is dense in $S(G)$.
\end{lem}
\begin{proof}
    Consider, for any cell $\sigma\in W_\Fsing$, the restriction map
    \[\mathrm{res|_{G_\sigma}}:S(G)\to S(G_\sigma).\]
Then $\Omega_\sigma$ is precisely the preimage under $\mathrm{res|_{G_\sigma}}$ of the open set $\Sigma^n(G_\sigma,\Z)$. As $\mathrm{res|_{G_\sigma}}$ is continuous, we deduce that $\Omega_\sigma$ is open. The same happens for the set $-\Omega_\sigma$. Moreover, we are assuming that $\Omega_\sigma$ is dense, so also $-\Omega_\sigma$ is. As $W_{\Fsing}$ has a finite number of cells we see that the set
 \[\Omega:=\{[\chi]\in S(G)\mid [\chi|_{G_\sigma}]\in\Sigma^n(G_\sigma,\Z)\cap-\Sigma^n(G_\sigma,\Z)\ \text{for every cell }\sigma\in W_{\Fsing}\}=\bigcap_{\sigma\in W_\Fsing}(\Omega_\sigma\cap-\Omega_\sigma)\]
 is the intersection of finitely many open and dense subsets, so $\Omega$ itself is open and dense. Finally, as the set of discrete characters in $S(G)$ is also dense and $\Omega_\Z$ is the intersection of that set with $\Omega$, we get the result.

\end{proof}

The condition that the sets $\Omega_\sigma$ are dense might be difficult to check, so the next Lemma will be useful in many cases. 

\begin{lem}\label{lem:densitycond} Let $H\leq G$ and $\pi:G\to G_\mathrm{ab}$ the abelianization map. The set
\[\Omega_H=\{[\chi]\in S(G)\mid[\chi|_{H}]\in\Sigma^n(H,\Z)\}\]
is dense in $S(G)$ in any of the following cases.
\begin{itemize}
    \item[(i)] if $\Sigma^n(H,\Z)=S(H)$ and $\pi(H)$ is infinite, or

    \item[(ii)] if $\pi(Z(H))$ is infinite, or

    \item[(iii)] if $\Sigma^n(H,\Z)\subseteq S(H)\neq\emptyset$ is dense and $\pi|_H$ has finite kernel.
\end{itemize}
\end{lem}
\begin{proof}
    In case (i),  \[\mathrm{Ann}(H)=\{[\chi]\in S(G)\mid\chi(H)=0\}\]
    is the intersection with $S(G)$ of a proper linear subspace and $\Omega_H=S(G)-\mathrm{Ann}(H)$ so $\Omega_H$ is dense.

Case (ii) is similar, because now $\mathrm{Ann}(Z(H))$ is the intersection with $S(G)$ of a proper linear subspace and Lemma \ref{lem:SigmaCentre} implies
\[S(G)-\mathrm{Ann}(Z(H))\subseteq\Omega_H.\]

Finally, assume (iii). As $\pi|_H$ has finite kernel, the restriction map 
  \[\mathrm{res|_{H}}:S(G)\to S(H)\]
  is surjective so there is an inverse $\tau:S(H)\to S(G)$ which is continuous. Therefore the image under $\tau$ of $\Sigma^n(H,\Z)$ is dense in $S(G)$. As 
  \[\tau(\Sigma^n(H,\Z))=\Omega_H\]  
  we are done.
\end{proof}

Some examples of groups $G$ so that $\Sigma^n(G,\Z)=S(G)$ include:
\begin{itemize}
\item infinite polycyclic-by-finite groups (by \cite{Hall} they are of type $\mathrm{FP}_\infty$, and all their subgroups are also of type $\mathrm{FP}_\infty$; hence by Theorem \ref{teo:FinitenesSigma} their Sigma invariants equal the whole sphere),  
\item spherical Artin groups of types $\mathbb{A}_n$, $\mathbb{D}_n$, $\mathbb{E}_6$, $\mathbb{E}_7$, $\mathbb{E}_8$, $\mathbb{H}_3$, and $\mathbb{H}_4$ (c.f.\ \cite[Theorem~1.4]{Escartin-Martinez}).  
\end{itemize}

We are ready now to show Theorem \ref{teo:A} from the introduction

\bigskip


\noindent{\sl Proof of Theorem \ref{teo:A}.} Let $\Omega_\Z$ be the set from Lemma \ref{lem:density}. Our hypothesis imply that $\Omega_\Z$ is dense in $S(G)$.

Assume that (v) holds. Then Theorem \ref{teo:charsigma} implies that $\Omega \subseteq \Sigma^n(G, \mathbb{Z})$, and since $\Omega_\Z$ is dense, we obtain (i). If $\Sigma^n(G,\Z)$ is dense, so is $-\Sigma^n(G,\Z)$ and as both are open, we see that (i) implies (ii).
Clearly, (ii) implies (iii). And if we have (iii), as $\Sigma(G,\Z)\cap-\Sigma^n(G,\Z)$ is open and the set of discrete characters is dense, we must have some discrete $[\chi]\in\Sigma(G,\Z)\cap-\Sigma^n(G,\Z)$. Using Theorem \ref{teo:FinitenesSigma} for $N=\ker(\chi)$ implies that $\ker(\chi)$ is of type $\FP_n$ so we get (iv). The fact that (iv) implies (iii) is obvious.

Finally, if (iii) holds, as $\Sigma^n(G,\Z)\cap-\Sigma^n(G,\Z)$, is open and non empty we get  
\[\Omega_\Z\cap\Sigma^n(G,\Z)\cap-\Sigma^n(G,\Z)\neq\emptyset\]
so there is some discrete $[\chi] \in \Omega$ as in Theorem \ref{teo:charsigma} and we deduce (v).
\qed

\bigskip



As the ``if'' part of Theorem \ref{teo:charsigma} is also true in the homotopical setting, we also deduce the following homotopical weaker version of Theorem \ref{teo:A}.

\begin{teo}\label{teo:homotopicA}
    Let $G\curvearrowright X$ be as in Hypothesis (*), and assume that $W_{\Fsing}$ is simply connected and that for each cell $\sigma\in W_\Fsing$, the set
    \[\Omega^t_\sigma=\{[\chi]\in S(G)\mid[\chi|_{G_\sigma}]\in\Sigma^n(G_\sigma)\}\]
 is dense in $S(G)$. Then the following are equivalent:
    \begin{itemize}
     \item[(i)] $\Sigma^n(G)$ is dense in $S(G)$,
     \item[(ii)] $\Sigma^n(G)\cap -\Sigma^n(G)$ is dense in $S(G)$,
    \item[(iii)] $\Sigma^n(G)\cap -\Sigma^n(G)\neq\emptyset$,
    \item[(iv)] $G$  fibers homotopically at degree $n$,
    \item[(v)] $W_{\Fsing}$ is $(n-1)$-connected.
\end{itemize}
\end{teo}
\begin{proof} As in the proof of Theorem \ref{teo:A} and with the same argument as in Lemma \ref{lem:density}, the hypothesis implies that
 \[\Omega^t_\Z:=\{[\chi]\in S(G)\ \text{discrete}\mid [\chi|_{G_\sigma}]\in\Sigma^n(G_\sigma)\cap-\Sigma^n(G_\sigma)\ \text{for every cell }\sigma\in W_{\Fsing}\}\]
 is dense in $S(G)$. 
If (v) holds, the `ìf" part of Theorem \ref{teo:charsigma} implies that $\Omega^t_\Z \subseteq \Sigma^n(G)$, so we get (i). The following implications
\[\mathrm{(i)}\Rightarrow\mathrm{(ii)}\Rightarrow\mathrm{(iii)}\iff\mathrm{(iv)}\]
are exactly as in the proof of Theorem \ref{teo:A}.

Finally, assume (iii). As the homotopical invariants always lie inside the homotopic invariants, we deduce that also
\[\Sigma^n(G,\Z)\cap-\Sigma^n(G,\Z)\neq\emptyset.\]
Moreover, the hypothesis on the density of $\Omega^t_\sigma$ for each cell $\sigma\in W_\Fsing$ implies that also the sets $\Omega_\sigma$ are dense. This means that we can use  Theorem \ref{teo:A} and deduce that $W_{\Fsing}$ is $(n-1)$-acyclic.
By Hurewicz Theorem, as $W_\text{sing}$ is simply connected, it is also $(n-1)$-connected so we get (v). 
\end{proof}

\section{Fouxé-Rabinovitch automorphism groups of free products}\label{sec:Fouxe}

Let $G_1,\ldots,G_n$ be groups and consider their free product $S = G_1 \ast \cdots \ast G_n$. The Fouxé-Rabinovitch automorphism group $\mathrm{FR}(S)$ is the
subgroup of $\Aut(S)$ generated by the automorphisms given by:

$$\left\{\begin{array}{l}
C_{i}^v(w)=vwv^{-1}, \text{ if } w\in G_i \\
C_i^v(z)=z, \text{ if } z \notin G_i.
\end{array}\right.
$$
for distinct $i,j$ and some $v\in G_j$. These automorphisms are called \emph{partial conjugations}.
Since the subgroup $\mathrm{Inn}(G)$ of inner automorphisms lies inside $\mathrm{FR}(S)$, we can define the outer version as
$$
\mathrm{OFR}(S)=\mathrm{FR}(S)/\mathrm{Inn}(S).
$$
An important particular case of Fouxé-Rabinovitch groups are pure symmetric automorphisms of free groups. 
In \cite{MM}, McCullough--Miller constructed a complex on which the Fouxé-Rabinovitch automorphism group of a free product acts with nice properties. In this section, we introduce this complex, observe that this action satisfies Hypothesis (*) and give conditions under which Theorem \ref{teo:A} can be applied.

\begin{defi}\label{def:bipartite}
A bipartite labeled tree on $[n]=\lbrace 1,\dots,n\rbrace$ is a tree with $n$ vertices labeled by the integers from $1$ to $n$ and some finite number of unlabeled vertices, satisfying:
\begin{enumerate}
    \item Each edge connects a labeled vertex to an unlabeled vertex.
    \item Each unlabeled vertex has valence at least $2$.
\end{enumerate}
\end{defi}
We now define a partial order on the set of all bipartite labeled trees on $[n]$. If $T$ is a bipartite labeled tree on $[n]$ having a labeled vertex $v$ linked to two unlabeled vertices $w_1,w_2$, we can form a new tree $T'$ by identifying $w_1$ with $w_2$ and the edge joining $w_1$ to $v$ with the edge joining $w_2$ to $v$ . We say that $T'$ is obtained from $T$ by folding at $v$. We write $T_1\preceq T_2$ if $T_1$ can be obtained from $T_2$ by a (possibly empty) sequence of foldings. We get a poset $\Wh_n$ called the \emph{Whitehead poset}. The star with $n+1$ vertices whose unique unlabeled vertex is the center is denoted by $\OO_n$, one can check that $\OO_n$, also called the \emph{nuclear element}, is the minimal element of this poset. Let $\Wh_n^0$ be the poset that we obtain when we remove the nuclear element, i.e. $\Wh_n^0=\Wh_n\setminus\lbrace \OO_n\rbrace$. In \cite{BMcCMM}, Brady, McCamond, Meier and Miller show that the geometric realization of $\Wh_n^0$ is Cohen-Macaulay which has the following consequence:

\begin{prop}\cite[Theorem 5.13]{BMcCMM}\label{Whitehead conectivity}
The geometric realization of $\Wh_n^0$ is $(n-4)$-connected.
\end{prop}

The Whitehead poset is the main ingredient in the construction of another space: the McCullough-Miller complex $\MM_S$. The complex $\MM_S$ was   defined in \cite{MM} by gluing copies of $\Wh_n$ using certain equivalence relationship. As shown in \cite[Section 4]{MM}, it is a contractible complex where the group of $\mathrm{OFR}(S)$ acts and by construction, the geometric realization  $|\Wh_n|$ is  a  contractible strong fundamental domain of the action, so  Hypothesis (*) is satisfied. It is important to observe that although the whole McCullough-Miller complex $\MM_S$ does depend on the free factors $G_i$, the poset $\Wh_n^0$ does not, it depends only on the number $n$ of factors. 

As shown in \cite{Griffin}, an alternative way to define the complex $\MM_S$ is as the geometric realization $C(\mathcal{H})$ of the coset poset of certain family $\mathcal{H}$ of subgroups of the form $\prod_{j=1}^k G_{i_j}$ for some $i_j\in\lbrace 1,\dots,n\rbrace$, up to conjugacy these subgroups are precisely the stabilizers of the action  and one has $|\Wh_n|_\Fsing=|\Wh_n^0|$. 

To apply Theorem \ref{teo:A}, it will be useful to understand the abelianization of the group $\mathrm{OFR}(S)$, for which we use a standard presentation.

\begin{teo}[\cite{Gilbert}]\label{teo:presFR}
Let $S=G_1\ast\dots\ast G_n$. Then $\mathrm{FR}(S)$ is the group generated by the automorphisms $C_i^v$ for $i=1,\dots,n$, with $j\neq i$ and $1\neq v\in G_j$ subject to the following relations:
\begin{enumerate}
    \item $C_i^vC_i^w=C_i^{vw}$ for $v,w\in G_j$.
   \item $[C_i^v,C_k^w]=1$ for $v\in G_j,w\in G_l$ and $i\neq j,i\neq k,l\neq k,$.
    \item $[C_i^vC_k^v,C_i^w]=1$ for $v\in G_j,w\in G_k$ and distinct $i,j,k$.
\end{enumerate}
\end{teo}

Using this presentation we have

\begin{lem}\label{lem:stabFR} Let $S=G_1\ast\dots\ast G_n$ be a free product and $G=\mathrm{OFR}(S)$.
    For any cell $\sigma\in|\Wh_n|_\Fsing=|\Wh_n^0|$, the induced map
    \[(G_\sigma)_\mathrm{ab}\buildrel{\pi}\over\rightarrow G_\mathrm{ab}\]
    is injective.
\end{lem}
\begin{proof} Let $\overline{S}=(G_1)_\mathrm{ab}\ast\dots\ast(G_n)_\mathrm{ab}$, observe that the kernel of the natural map $S\to\overline{S}$ is preserved by $G$ so there is a map $G=\mathrm{OFR}(S)\to\overline{G}:=\mathrm{OFR}(\overline{S})$.
It is a consequence of the construction of the Whitehead poset $\Wh_n^0$ that if $G_\sigma=\prod_{j=1}^k G_{i_j}$, then $\overline{G}_\sigma=\prod_{j=1}^k (G_{i_j})_\mathrm{ab}$. Moreover, using the presentation of $\overline{G}$ given by Theorem \ref{teo:presFR} we see that the map $(\overline{G}_\sigma)_\mathrm{ab}\buildrel{\overline{\pi}}\over\rightarrow \overline{G}_\mathrm{ab}$ is injective. We have a commutative diagram

$$\begin{tikzcd}
&(G_\sigma)_\mathrm{ab}\arrow[r,"\pi"]\arrow[d,"1_d"]&G_\mathrm{ab}\arrow[d]\\
&(\overline{G}_\sigma)_\mathrm{ab}\arrow[r,"\overline{\pi}"]&\overline{G}_\mathrm{ab}\\
\end{tikzcd}$$
    which implies the result.
\end{proof}

As a consequence, using Theorem \ref{teo:A} we get:

\begin{cor}\label{cor:FR}
Let $S=G_1\ast\dots\ast G_n$ be a free product of $n\geq 4$  groups of type $\mathrm{FP}_\infty$ such that $\Sigma^n(G_i,\mathbb{Z})\subseteq S(G_i)$ is dense for all $i=1,\dots,n$. Then

Then $\Sigma^{n-3}(\mathrm{OFR}(S))$ is dense in $S(G)$, and $\Sigma^{n-2}(\mathrm{OFR}(S))\cap-\Sigma^{n-2}(\mathrm{OFR}(S))=\emptyset$.
\end{cor}

\begin{proof} The fact that we can apply Theorem \ref{teo:A} follows from the previous discussion: the complex $\MM_S$ is contractible and has strong fundamental domain $\Wh_n$ and for each cell $\sigma\in\Wh_n$ the set $\Omega_\sigma$ is dense because of Lemma \ref{lem:densitycond} (iii) and Lemma \ref{lem:stabFR}. 
Using Theorem \ref{teo:A} the result follows by Proposition \ref{Whitehead conectivity} 
\end{proof}




\section{Symmetric automorphisms of right angled Artin groups}\label{sec:MM}

\subsection{The \texorpdfstring{$\G$}-McCullough-Miller complex} The construction of the McCullough-Miller space for pure symmetric automorphisms of free products has been recently generalized by the first and third named authors to the case of pure symmetric automorphisms of right-angled Artin groups. In this section, we will briefly recall this generalization to show that this action also satisfies Hypothesis (*) and that Theorem \ref{teo:A} applies.

\begin{defi}[\textup{RAAG}]
    Let $\G$ be a finite simplicial  graph (i.e., with no double edges and no loops). The right-angled Artin group (\textup{RAAG}) defined by $\G$ is the group with presentation
    $$
A_{\Gamma}=\langle v \in \Gamma \mid[v, w]=1 \textrm{ whenever } \{v,w\} \in E(\Gamma)\rangle
$$
\end{defi}

Here and throughout the paper we will abuse notation and often identify a graph with its set of vertices. We will also use standard graph notation, for example, for a vertex $v\in\G$, $\lk(v)$ is the subgraph spanned by the vertices in $\G$ linked to $v$ and $\st(v)=\lk(v)\cup\{v\}$. All subgraphs $\Delta\subseteq\G$ are supposed to be \emph{induced}, meaning that any two vertices of $\Delta$ linked by an edge of $\Gamma$ are also linked in $\Delta$.

\begin{defi} Let $A_\Gamma$ be a \textup{RAAG}. The group of pure symmetric automorphisms $\PAut(A_\G)$ of $A_\G$ is the group consisting of those automorphisms of $A_\G$ that map each standard generator of $A_\G$ (i.e., each vertex of $\G$) to a conjugate of itself. Obviously, it contains the inner automorphisms and the corresponding external version is denoted $\POut(A_\G)=\PAut(A_\G)/\Inn(A_\G).$
\end{defi}

In the particular case when $\G$ is a disjoint union of complete graphs $\G=\G_1\sqcup\ldots\sqcup\G_k$, it is easy to check that $\POut(A_\G)=\mathrm{OFR}(A_{\G_1}\ast\ldots\ast A_{\G_k})$ so we can apply Corollary
\ref{cor:FR}. However, this is no longer true if the graphs $\G_1,\ldots,\G_k$ are not complete. Next, we introduce an important family of elements of $\PAut(A_\G)$.

\begin{defi}[Partial conjugations]
Given $v \in \G$ and a union $S$  of connected components of $\Gamma-\st(v)$, the partial conjugation of $S$ by $v$ is the automorphism $C_{S}^v$  given by
$$
\left\{\begin{array}{l}
C_{S}^v(w)=vwv^{-1}, \text{ if } w \in S \\
C_{S}^v(z)=z, \text{ if } z \in\Gamma-S.
\end{array}\right.
$$
\end{defi}

\begin{teo}[Laurence \cite{Lau}] The set of partial conjugations is a generating system for $\PAut(A_\G)$.   \label{teo:partialGen}
\end{teo}

Given a vertex $v\in\G$ we will refer to the connected components of $\G-\st(v)$ by saying simply \emph{components} of $\G-\st(v)$. If $w\in\G$ is another vertex not linked to $v$, the components of $\G-\st(v)$ can be classified with respect to $w$ as follows.

\begin{defi}[Graph components: shared, dominant and subordinate]
     Let $u, v\in \Gamma$ be not linked.  The unique component of $\Gamma-\st(v)$ that contains $w$ is called the \emph{dominant} component of $\G-\st(v)$ respect to $w$. A component of $\Gamma-\st(v)$ that is also a component of $\Gamma-\st(w)$ is said to be \emph{shared}, if this happens then we say that $v$, $w$ form a \emph{SIL-pair}. All the other components are called \emph{subordinate}. 
     \label{defi:comp}
\end{defi}
One can show that subordinate components of $\G-\st(v)$ are subsets of the dominant component of $\G-\st(w)$.
The notion of SIL pair appeared already in \cite{ChRSV} and the above classification of components goes back to \cite{GPR} and has been extensively used, see for example \cite{DayWade}. 
Using this terminology, it is possible to rewrite the presentation of Koban and Piggott (\cite{KoPi}) as follows:
\begin{teo}\label{teo:conmu}
    The group $\PAut(A_{\Gamma})$ is the group generated by all partial conjugations $C_{A}^v$ for $v\in \G$ and $A$ a component of $\G-\st(v)$, subject to the following relations:
 \begin{itemize}
     \item[(i)] $\left[C_{A}^v, C_{B}^w\right]=1$ if  $v \in \st(w)$,
     \item[(ii)]$\left[C_{A}^v, C_{B}^w\right]=1$ if $v \notin \st(w)$ and either $A$ and $B$ are shared but distinct or $A$ or $B$ are subordinate,
     \item[(iii)] $\left[C_{A}^v C_{B}^w, C_{A}^w\right]=1$ if $v \notin \st(w)$, $A$ is shared and $B$ dominant. \label{teo:presentation}
 \end{itemize}
\end{teo}


\begin{defi}[Based partitions, petals, length and automorphisms carried by a based partition]
Let $u\in \G$. A \emph{$\G$-partition based at $u$} or simply a \emph{based partition} is a partition of of the set $\G-\lk(u)$ of the form $\tau_u=\{\{u\},P_1,\ldots,P_k\}$  so that each $P_i$ is a union of components of $\G-\st(u)$. The sets $P_i$ are called \emph{petals} of $\tau_u$.   The number of petals is the 
    \emph{length} of the based partition, and we denote it by $l(\tau_u)$. The \emph{trivial} based partition at $u$ is $\{\{u\},\G-\st(u)\}$, it is the only partition based at $u$ of length 1.   \label{defi:partitions}
\end{defi}


\begin{defi}[Crossings and compatible partitions]
We say that two based partitions $\tau_u$ and $\tau_v$ \emph{cross} if $u\notin \st(v)$ and there are petals $P$ of $\tau_u$ and $Q$ of $\tau_v$ such that  $v\not\in P$, $u\not\in Q$, and $P\cap Q\neq\emptyset$, i.e., $P$ and $Q$ do not contain the corresponding dominant components and there is at least one shared component $\C$ such that $\C\subseteq P\cap Q$. 
 We will say that two based partitions $\tau_u, \tau_v$  are \emph{compatible} if either $u$ and $v$ are linked in $\G$ or they are not linked but $\tau_u, \tau_v$ do not cross.
\end{defi}

An important consequence of this definition is that if the vertices $u,v$ do not from a SIL pair, then any pair of partitions $\tau_u$, $\tau_v$ based at $u,v$ are always compatible.


The construction of the McCullough-Miller space has as main ingredient the Whitehead poset whose elements are the labeled bipartite trees of Definition \ref{def:bipartite}. In the case of the generalized McCullough-Miller space we will also use a generalized Whitehead poset, in this case the elements of the poset are called vertex types.

\begin{defi}[Vertex type]\label{def:vertextype}
    We define a vertex type $\tau=\{\tau_{a}\}_{a\in\G}$ as a collection of pairwise compatible based partitions, one for each vertex in $\G$. We say that two vertex types $\tau$ and $\tau'$ are \emph{compatible} if their based partitions are pairwise compatible and that a collection of vertex types is \emph{compatible} if they are pairwise compatible. When a vertex type has only a non trivial based partition based at, say $a\in\G$, we will often use the notation $\tau_a$ for both the vertex type and the based partition.
\end{defi}

 The \emph{rank} of the vertex type $\tau$ is the sum of the lengths minus one, i.e.
   $$\mathrm{rk}(\tau)=\displaystyle\sum_{a\in\G}(l(\tau_a)-1)$$


\begin{defi}[$\Gamma$-Whitehead poset]\label{def:order}  
    Let $v\in\Gamma$ and $\tau_v$, $\tau'_v$ two partitions based at $v$.  We say that $\tau_v\leq \tau_v'$ if every petal of $\tau_v'$ is contained in a petal of $\tau_v$, i.e., when every petal of $\tau_v$ is a union of petals of $\tau_v'$. In particular, if $\tau_v\leq \tau_v'$, then $l(\tau_v)\leq l(\tau'_v)$.  For vertex types, we will say that $\tau\leq \tau'$ if $\tau_v\leq \tau'_v,\ \textrm{ for any } v\in \Gamma$.
    The $\Gamma$-Whitehead poset $\Wh_{\G}$ is the poset formed by all the vertex types associated to $\G$ under this partial order. It has a minimal element called the \emph{nuclear vertex type} $\OO$, whose based partitions are all trivial. 
\end{defi}

The order relation implies compatibility, i.e., if $\tau$ and $\tau'$ are vertex types with $\tau<\tau'$, then $\tau$ and $\tau'$ are compatible (\cite[Lemma 4.5.]{ArMar}). Moreover, compatibility implies the existence of upper bounds:

\begin{lem}\cite[Lemma 5.6]{ArMar}\label{lem:lub} Compatible vertex types $\tau,\tau'$  have a least upper bound $\tau\vee\tau'$ which is formed as follows: for each $a\in\Gamma$, consider the based partitions $\tau_a=\{\{\underline{a}\},P_1,\ldots,P_t\}$, $\tau_a'=\{\{\underline{a}\},Q_1,\ldots,Q_s\}$. Then the based partition $(\tau\vee\tau')_a$ has as petals all the non empty intersections $P_i\cap Q_j$. 
\end{lem}

\begin{defi}[The complex $\MM_\G$]
    The McCullough-Miller complex $\MM_\G$ for a \textup{RAAG} $A_\Gamma$ is a CW-complex constructed by gluing copies of $|\Wh_\G|$, the geometric realization of the $\G$-Whitehead poset, for details on the gluing see Section 4 of \cite{ArMar}.
\end{defi}

By Theorems A and B of \cite{ArMar}, $\MM_\G$ is a contractible space where the group $\POut(\A)$ acts with the geometric realization $|\Wh_\G|$ as  strong fundamental domain. As a consequence, Hypothesis (*) holds for the action $\POut(\A)\curvearrowright \MM_\G$. Those theorems also imply that the stabilizers of the action are free abelian, and can be described rather explicitly as we shall do next.

Let $\tau$ be a vertex type, seen as a 0-cell in $\MM_\G$. By the Remark after Lemma 4.13 in \cite{ArMar},  the stablizer of $\tau$ under the action of $\POut(\A)$ is precisely the group consisting of the so called external automorphisms \emph{carried} by $\tau$. For each based partition $\tau_a$, the group of external automorphisms carried by $\tau_a$ is the group
\[\Stab(\tau_a)=\langle C^a_P\mid P\text{ petal of }\tau_a\rangle\]
and the group of external automorphisms carried by the vertex type $\tau$ is
\[\Stab(\tau)=\langle \Stab(\tau_a)\mid a\in\G\rangle.\]
The group $\Stab(\tau)$ is free abelian of rank $\mathrm{rk}(\tau)$ and the group $\Stab(\OO)$ is trivial. Crucially, these groups can be used to characterize when two vertex types are compatible:

\begin{lem}\cite[Lemma 5.4]{ArMar}\label{lem:comp}
     Two vertex types $\tau,\tau'$ are compatible if and only if the groups $\Stab(\tau)$ and $\Stab(\tau')$ commute.
\end{lem}

 The following fact follows from the definition and will be useful below:

\begin{rem}\label{rem:stab} The vertex type $\tau$ can be recovered from the group $\Stab(\tau)$ and given two vertex types $\tau,\tau'$ such that $\Stab(\tau)\leq\Stab(\tau')$, we have $\tau\leq\tau'$. To see it note first that using the definition of $\Stab(\tau)$ we may reduce the problem to the case when $\tau$ has only one non trivial based partition at, say $a\in\G$. Then the petals of $\tau_a$ are determined by the partial conjugations in $\Stab(\tau)$.
\end{rem}

Moreover, the set of cell stabilizers for the action $\POut(\A)\curvearrowright \MM_\G$ is precisely the set of $\POut(\A)$-conjugated subgroups of the groups $\Stab(\tau)$ where $\tau$ runs over the vertex types. The groups $\Stab(\tau)$ are generated by partial conjugations and for any $a\in\G$ and $A=U\cup V\subseteq\G$ with $U,V$ union of components of $\G-\st(a)$ we have $C^a_A=C^a_UC^a_V$, so the generators of $\Stab(\tau)$ are products of the standard generators of $\POut(\A)$.
Therefore the presentation of Theorem \ref{teo:conmu} implies that the induced map 
\[\Stab(\tau)=(\Stab(\tau))_\mathrm{ab}\to(\POut(A_\G))_\mathrm{ab}\]
is injective. This together with  Lemma \ref{lem:densitycond} implies that the hypotheses of Theorem \ref{teo:A} hold true, in fact the three conditions (i), (ii) and (iii) of Lemma \ref{lem:densitycond} are satisfied. The subcomplex $|\Wh_\G|_\Fsing$ is precisely the simplicial realization of the following subposet.

\begin{defi}[$\Wh_\G^0$]
    We define $\Wh_\G^0$ as the poset that we obtain  when removing the nuclear vertex type $\OO$ in the Whitehead poset $\Wh_\G$.
\end{defi}

One more consequence of the presentation in Theorem \ref{teo:conmu} is that the Sigma invariants of the group $\POut(A_\G)$ are symmetric. Therefore
we get the following Theorem which is part of the statement of Theorems \ref{teo:families} and \ref{teo:POutRAAGs} in the introduction.

\begin{teo}\label{teo:firstPartPOutRAAGs}
    Let $G=\POut(\A)$ be the pure symmetric outer automorphism group of a \textup{RAAG} $\A$. Then, for any $k$, $\Sigma^k(G,\Z)$ is  dense in $S(G)$ if and only if  $|\Who|$ is $(k-1)$-acyclic and in other case it is empty.
\end{teo}


\subsection{ The first invariant \texorpdfstring{$\Sigma^1$}\/ and the connectivity of the fundamental domain }

 The particular case when $k=1$ in Theorem \ref{teo:firstPartPOutRAAGs} implies that, for the group $G=\POut(\A)$, the invariant $\Sigma^1(G)$ is either empty or dense in $S(G)$. This is not new: it is a consequence of the explicit description of $\Sigma^1$ for these groups given in \cite{DayWade}, where it is shown that $\Sigma^1$ is the result of removing certain linear subspaces from $S(G)$. Moreover, Theorems 1.3 and 4.3 in \cite{Escartin} characterize in terms of $\G$ which is the case. 
We recall below this characterization, to do that we first introduce the following notation. 

Let $\G$ be a graph such that for any $v\in\G$, $\G-\st(v)$ has at most two components. It follows from results in  \cite{DayWade}, cf. Theorem \ref{teo:DayWade} below, that this condition implies that $\POut(A_\G)$ is isomorphic to a RAAG. This is true for more general graphs, but in this case, the description of the defining graph of this new RAAG is particularly easy.  
Let $\G^\sil$ be the  graph having as vertices all those $v\in\Gamma$ such that $\G-\st(v)$ has precisely 2 components, and an edge between two such vertices $v,w$ whenever $v$ and $w$ do not form a SIL pair in $\Gamma$. Then $\POut(A_\G)$ is the RAAG $A_{\G^\sil}$. Moreover, if we denote by $\widehat{\G^\sil}$ the flag complex obtained by gluing an $(n-1)$-simplex to each $n$-clique of $\G^\sil$,
it turns out that the barycentric subdivision of $\widehat{\G^\sil}$  can be identified with $|\Wh_{\G}^0|$.  This is in fact is a particular case of the more general Theorem \ref{teo:homotopWhi-RAAG} below, but we think it is worth to see this case first because the relation between both complexes is much more explicit here. We denote by $\mathcal{P}(\Gamma^\sil)$ the poset of (non-empty) cliques in $\Gamma^\sil$, so the barycentric subdivision of $\widehat{\G^\sil}$ is the geometric realization $|\mathcal{P}(\Gamma^\sil)|$.


\begin{teo}\label{teo:barycentricSubd}
    Let $\G$ be a graph such that, for every $v\in\G$, $\G-\st(v)$ has at most two components. Then the posets $\mathcal{P}(\Gamma^\sil)$ and  $\Wh_\G^0$ are isomorphic.
\end{teo}
\begin{proof} We define a map $\phi:\mathcal{P}(\Gamma^\sil)\to\Wh_\G^0$ as follows.
Let $\sigma=\{x_0,\ldots,x_k\}$ be a clique in $\Gamma^\sil$. Then for each $x_i$ there is a vertex type $\tau_{x_i}$ with only one non trivial based partition which is based at $x_i$ and the hypothesis implies that this based partition has exactly 2 petals. The fact that $\sigma$ is a clique implies that the elements $\{x_0,\ldots,x_k\}$ do not form SIL-pairs with each other and therefore, the vertex types $\{\tau_{x_0},\ldots,\tau_{x_k}\}$ are all compatible so there exists a lower upper bound $\tau_{x_0}\vee\ldots\vee\tau_{x_k}$. We set $\phi(\sigma)=\tau_{x_o}\vee\ldots\vee\tau_{x_k}$. The only non trivial based partitions of this vertex type are based precisely  at the vertices $x_0,\ldots,x_k$. Taking into account that the only possible non-trivial based partitions are of the form  $\phi(x)$ for $x\in\G^\sil$ one easily sees that $\phi$ is indeed a poset isomorphism. 
\end{proof}

Using this fact, we can state as follows the characterization of when $\Sigma^1(\POut(A_\G))$ is nonempty (the equivalence between (iv) and (v) is implicit in the proof of \cite{Escartin} but not explicitly stated there): 
 
\begin{teo}\cite[Theorems 1.3 and 4.3]{Escartin},\cite{DayWade}\label{teo:sigma1}
    Let $G=\POut(\A)$ be the pure symmetric outer-automorphism group of a \textup{RAAG} $\A$. Then, the the following are equivalent:
    \begin{itemize}
        \item[(i)] $\Sigma^{1}(G)\neq\emptyset$,

        \item[(ii)] $\Sigma^1(G)$ is dense in $S(G)$

        \item[(iii)] $|\Who|$ is connected,

        \item[(iv)] either there exists $v\in\G$ such that $\G-\st(v)$ has at least three connected components, or $\G-\st(v)$ has at most two connected components for every $v$ and the SIL graph $\G^\sil$ is connected,

        \item[(v)] $G$ fibers,

        \item[(vi)] the first $\ell^{(2)}$-Betti number of $G$ vanishes.
    \end{itemize}
\end{teo}

\section{The support graph}\label{sec:supp}

In \cite{DayWade}, Day and Wade define, for each vertex $v\in\G$, a graph called the \emph{support graph} that they use to characterize when the group $\POut(A_\G)$ is isomorphic to a RAAG. This graph turns out to be useful to give a condition that implies the contractibility of the complex $|\Who|$.

\begin{defi}[Support graph]
    For each $v\in\G$, we define a new graph called  the \emph{support graph} $\Delta_v$ with vertex set the set of components of $\G-\st(v)$. Given $S,T$  components  of $\G-\st(v)$; there is an edge $(S,T)$ in $\Delta_v$  if there exists some $w\in\G$ such that $w\in T$ and $S$ is a shared component for $v$ and $w$.
\end{defi}

Following \cite{DayWade}, we will use the next notation: for $v\in\G$ and a subgraph $L\subseteq\Delta_v$, $C^v_L$ is the partial conjugation
$$
C_L^v:=\prod_{A\in L}C_A^v.
$$
By \cite[Proposition 2.9]{DayWade}, if $L$ is a connected component of the support graph $\Delta_v$, then
the image of the partial conjugation $C_L^v$ in $\POut(A)$ is central. 
This has strong consequences because of Lemma \ref{lem:comp} as we shall see next.

\begin{prop}\label{teo:supportGraph}
     Assume that there is a vertex $v\in \G$ such that its support graph $\Delta_v$ has at least 2 connected components. Then $|\Wh_\G^0|$ is contractible.
\end{prop}
\begin{proof} Let $L_1,...,L_t$ be the connected components in the support graph $\Delta_v$, so the assumption implies $t\geq 2$. Let $\delta$ be the vertex type  with only non trivial based  partition $\delta_v=\{\{v\},L_1,...,L_t\}$. Then $\Stab(\delta)$ is generated by the images in $\POut(\A)$ of the partial conjugations
$$C^v_{L_1},\ldots,C^v_{L_t}$$
which, by \cite[Proposition 2.9]{DayWade}, are all central, so the group $\Stab(\delta)$ itself is central. By
Lemma \ref{lem:comp}, this implies that $\delta$ is compatible with any other vertex type $\tau\in\Wh_\G$ so we have a poset map:
   $$
f: \Wh_\G^0\to\Wh_\G^0
   $$
   $$
\tau\mapsto \tau\vee\delta.
   $$
   As $\delta\leq f(\tau) \geq \tau$ for all vertex types $\tau$,  Quillen poset Lemma (cf. \cite[Lemma 1.5]{Q}) imply that $|\Wh_\G^0|$ is contractible.
\end{proof}

In the particular case when the defining graph $\G$ is a tree this condition is easy to check, mainly because then we have the following simple way to determine when two vertices form a SIL pair.

\begin{rem}\label{rem:siltree} Let $\A$ be a \textup{RAAG} such that the defining graph $\G$ is a tree. Then two vertices $a,v\in\G$ form a SIL pair if, and only if, there is a third vertex $w$ linked to both $a,v$ that has valence at least three. In that case the shared components are precisely the components of $\G-w$ that do not contain neither $a$ nor $v$.
\end{rem}

\begin{figure}
\begin{tikzpicture}
   \centering
      
 \draw (0,0)--(3,0);
 \draw (1,0)--(2/3,2/3);
  \draw (1,0)--(4/3,2/3);
   \draw (2,0)--(7/3,-2/3);
    \draw (2,0)--(8.4/3,4/3);
    \draw (7/3,2/3)--(10/3,2/3);
\node[label={ below: $v$}] at (0,0){$\bullet$}; 
\node[label={ below: $w$}] at (1,0){$\bullet$}; 
\node[label={ below: $a$}] at (2,0){$\bullet$}; 
\node at (3,0){$\bullet$}; 
\node[label={above: $C_1$}] at (2/3,2/3){$\bullet$}; 
\node[label={above: $C_1$}] at  (4/3,2/3){$\bullet$}; 
\node at (7/3,-2/3){$\bullet$}; 
\node at (7.2/3,2/3){$\bullet$}; 
\node at (10/3,2/3){$\bullet$}; 
\node at (8.4/3,4/3){$\bullet$}; 
    \end{tikzpicture}
    \caption{SIL pair $\{v,a\}$ with shared components $C_1$ and $C_2$}
    \end{figure}
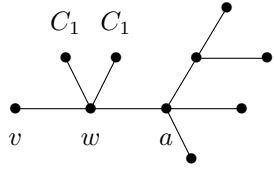
\smallskip

Using this observation we have:

\begin{lem}\label{lem:3notLeaves}
    Let $\A$ be a \textup{RAAG} with $\G$ a tree having at least three nodes which are not leaves. Then there is some $v\in\G$ such that its support graph $\Delta_v$ has at least two connected components so $|\Wh_\G^0|$ is contractible. 
\end{lem}
\begin{proof} The hypothesis on $\G$ implies that there is some vertex $v\in\G$ such that the set $\Omega$ of non-leaves in $\lk(v)$ contains at least two elements. Observe first that for any $w\in\Omega$, as $w$ is not a leave, there must be some component of $\G-\st(v)$ having some vertex linked to $w$, and as $\G$ is a tree, any such component can not be linked to two different elements of $\Omega$.

Now, let $(P,Q)$ be an edge of $\Delta_v$. Then, upon a possible interchange of $P,Q$, there must be some $a\in\G$ such that $a,v$ form a SIL pair so that $P,Q$ are components of $\Gamma-\st(v)$, $a\in P$ and $Q$ is shared for $a$ and $v$. 
By Remark \ref{rem:siltree}, this implies that $a$ must be linked to some of the elements $w\in\Omega$ and that some vertex in $Q$ is also linked to $w$. This means that we can use $\Omega$ to distinguish path connected components of the support graph. As $\Omega$ has at least 2 elements, we deduce that $\Delta_v$ has at least two connected components so by Theorem \ref{teo:supportGraph}, $|\Wh_\G^0|$ is contractible.
\end{proof}

At this point, we can easily give a complete description of the homotopy type of $|\Who|$ in the case when $\G$ is a tree, because the only possible trees which do not satisfy the hypothesis of Lemma \ref{lem:3notLeaves} are either stars or consist of two stars linked by the internal node. In the first case, i.e., if $\Gamma$ is a star of, say, $n$ leaves, then it is easy to see that the central vertex can be removed without altering the $\POut(A_\Gamma)$, and in fact $\Wh_\G=\Wh_n$. Finally, to understand the homotopy type of $\Who$ when $\Gamma$ consists of two linked stars we will use the next more general fact.

\begin{prop}\label{prop:join}
    Let $\A$ be a \textup{RAAG} and assume that we can split the set of vertices of $\G$  as a disjoint union
    \[V(\Gamma)=V_1\sqcup\ldots\sqcup V_k\]
    such that for any $v\in V_i$ and $w\in V_j$, with $i\neq j$, $v$ and $w$ do not form a SIL-pair. Let $W_i$ be the subposet of $\Wh_\G$ formed by all the vertex types in $\Wh_\G$ with all the based partitions trivial except of possibly those based at vertices in $V_i$ and $W_i^0=W_i-\OO$. Then there is a homotopy equivalence
    \[|\Wh^0_\G|\simeq |W_1^0| *\cdots* |W^0_k|.\]

\end{prop}

\begin{proof} The hypothesis implies that the vertex types in $W_1,\ldots,W_k$ are all compatible, so there is a poset map
   \begin{align*}
       \varphi:W_1^0\times\cdots\times W_k^0&\to \Who\\
\ \ (\gamma^1,...,\gamma^k)\ \  &\mapsto \gamma^1\vee\ldots\vee\gamma^k\\
   \end{align*} 
which is clearly an isomorphism, because for any $\tau\in\Who$ we may write $\tau=\tau^1\vee\ldots\vee\tau^k$ where $\tau^i$ has the same based partition based at vertices in $V_i$ as $\tau$, and trivial based partitions at any other vertex. The result follows by iterating a Lemma due to Quillen (Proposition 1.9 in \cite{Q}), who shows that given posets $P,Q$ with minimal element $0_P,0_Q$ respectively, there is an homotopy equivalence
    \[
    P\times Q \backslash\{(0_P,0_Q)\}\simeq P \backslash\{0_P\}*Q \backslash\{0_Q\}.
    \]
    \end{proof}

\begin{cor}\label{cor:twoStars}
    Let $\A$ be a \textup{RAAG} with $\G$ a tree consisting of two stars $S_1,S_2$ linked by an edge between the two internal nodes. If $S_1$ has $n_1$ leaves and $S_2$ has $n_2$, then $|\Wh_\G^0|\simeq |\Wh_{n_1+1}^0|*|\Wh_{n_2+1}^0|$. 
\end{cor}
\begin{proof}
    We claim that we can decompose $V(\G)=V_1\sqcup V_2$ as a disjoint union such that for any $v\in V_1$ and $w\in V_2$, $v$ and $w$ do not form a SIL-pair. 
    Let $r_1$ be the internal node of $S_1$ and $r_2$ the internal node of $S_2$.
    We set $V_1=V(\lk(r_1))$ and $V_2=V(\lk(r_2))$. It is clear that $V(\G)=V_1\sqcup V_2$. Let $v\in V_1$ and $w\in V_2$. If $v$ and $w$ are linked then $v=r_1$ and $w=r_2$, and they do not form a SIL pair. In other case, $v$ must be a leave in $S_1$ and $w$ must be a leave in $S_2$ so there is no vertex linked to both and by Remark \ref{rem:siltree} they can not form a SIL pair.  As a consequence, we can apply Corollary \ref{prop:join} and we get 
\[|\Wh_\G^0|\simeq |W_1^0|*|W_2^0|\]
where $W_i^0$ is the poset of non-trivial vertex types whose non trivial based partitions are based at vertices of $V_i$.
Finally, observe that $W_i^0=\Wh_{n_i+1}^0$ for $i=1,2$ because the vertex types in $W_i^0$ can be seen as the vertex types associated to the free group on $V_i$.
\end{proof}

We are now ready to show Theorem \ref{teo:POutRAAGs}.

\bigskip

\noindent{\sl Proof of Theorem \ref{teo:POutRAAGs}.} 
Part (i) is a consequence of Lemma \ref{lem:3notLeaves}, and part (iii) follows from Corollary \ref{cor:twoStars}.

\qed

\subsection{The case when \texorpdfstring{$\POut(\A)$} is a RAAG}\label{subsec:POutRAAG}

 As observed before, in  \cite{DayWade}, Day and Wade  use the notion of support graph to characterize when the group $\POut(\A)$ is itself a RAAG. We state this characterization next.

\begin{teo}[Theorem 5.12, \cite{DayWade}]\label{teo:DayWade}
    The group $\POut(\A)$ is isomorphic to a \textup{RAAG} if and only if the support graph $\Delta_a$ if a forest for each $a\in\G$. 
\end{teo}

For a RAAG $A_\Theta$, there is a well known action that satisfies Hypothesis (*), namely, the action on the Deligne complex. The Deligne complex can be seen as the geometric realization of the coset poset of \emph{special spherical subgroups}, which are the subgroups of $A_\Theta$ generated by \emph{cliques}, i.e., complete subgraphs $\sigma\subseteq\Theta$. This means that the singular subcomplex of the (strong) fundamental domain of the action is the flag complex $\widehat{\Theta}$ of the defining graph $\Theta$. 

Therefore, if $\POut(\A)$ is isomorphic to a RAAG, we could apply \ref{teo:A} either using the action on the Deligne complex, or else using the action on the McCullough-Miller complex. And as a consequence, we see that, for any $n\geq 0$, $|\Who|$ is $(n-1)$-acyclic if and only if $\widehat{\Theta}$ is. Therefore it is natural to ask whether the simplicial complexes $\widehat{\Theta}_\Fsing$ and $|\Who|_\Fsing$ are homotopy equivalent. In this subsection we are going to prove that this is indeed the case. A first observation is that the singular set $\widehat{\Theta}_\Fsing$ of the flag complex $\widehat{\Theta}$ is homotopy equivalent to the simplicial realization of the poset $\mathcal{P}(\Theta)$ of (non empty) cliques of $\Theta$.

We will need the following explicit description of the defining graph $\Theta$ of the RAAG on Theorem \ref{teo:DayWade}.

\begin{defi}[Definition 5.4, \cite{DayWade}]\label{def:DayWadeRAAG}
    Let $\G$ be a simplicial graph such that for each $a\in\G$ the support graph $\Delta_a$ is a forest. Fix, for each support graph $\Delta_a$:
    \begin{itemize}
    \item a preferred maximal tree  in the forest $\Delta_a$
    \item a base point in each of the maximal trees $\Delta_a$. We will assume that these base points are leaves of the corresponding trees.
    \end{itemize}
    We define now a family of partial conjugations of $A_\Gamma$ and a new graph $\Theta$. For each $a\in\G$ we consider
    \begin{itemize}
        \item[\bf{Type 1:}]  For each maximal tree $T$ in $\Delta_a$ and each edge $e$ of $T$, we have a vertex $v_e^a$ of $\Theta$ and a partial conjugation $C_{L(e)}^a$ where $L(e),L_0$ are such that $T-e=L(e)\cup L_0$ and the base point of $T$ lies in $L_0$.
    
        \item[\bf{Type 2:}] For each maximal tree $T$ of $\Delta_a$ which is not the preferred maximal tree, we have a vertex  $v_C^a$ of $\Theta$ and a partial conjugation $C_T^a$.
    \end{itemize}
    The edges of $\Theta$ are defined so that there is an edge between $x,y\in \Theta$ if and only if the images in $\POut(\A)$ of the associated partial conjugations commute.
\end{defi}

The RAAG of Theorem \ref{teo:DayWade} is precisely the RAAG with defining graph the graph $\Theta$ constructed in Definition \ref{def:DayWadeRAAG}. In what follows, we assume that $\G$ is a graph such that for any $a\in\G$ the support graph $\Delta_a$ is a forest and identify the groups $\POut(\A)$ and $A_\Theta$. We will use the following observation about special spherical subgroups in RAAGs.

\begin{lem}\label{lem:commuting} Let $A_\Theta$ be a \textup{RAAG} with free abelian subgroups $H_1,H_2\leq A_\Theta$ such that $[H_1,H_2]=1$ and both $H_1$ and $H_2$ lie inside special spherical subgroups. Let $\sigma_1,\sigma_2\subseteq\Theta$, be the smallest cliques such that $H_1$ resp. $H_2$ lie inside the special subgroup $A_1=A_{\sigma_1}$ resp. $A_2=A_{\sigma_2}$. Then 
$$[A_1,A_2]=1$$
\end{lem}
\begin{proof}
    Put $\sigma_1=\{v_1,\ldots,v_t\}$  and $\sigma_2=\{w_1,\ldots,w_s\}$. For each $g\in H_1$ we have a cyclically reduced expression
$$g=v_1^{r_1}\ldots v_t^{r_t}$$
and similarly for any element $h\in H_2$. As $g$ and $h$ commute, using the description of centralizers of elements given by Servatius in \cite{Servatius}, we deduce that the elements of $\sigma_1$ in the support of $g$ commute with the elements of $\sigma_2$ in the support of $h$. The minimality of $\sigma_1$ and $\sigma_2$ implies that every element of $\sigma_1$ is in the support of some $g\in H_1$ and the same for $\sigma_2$ so we are done.
\end{proof}

 Next, we define a poset map $\phi:\mathcal{P}(\Theta)\to\Who$ and check that it induces a homotopy equivalence. To do that we set, using the notation of Definition \ref{def:DayWadeRAAG}:
 \begin{itemize}
     \item For each $v_e^a\in\Theta$ of Type 1, $\tau(v_e^a)\in\Who$ is the vertex type whose unique non-trivial based partition is
    $$\tau(v_e^a)_a=\{\{a\},L(e),(\G-\st(a))-L(e)\}.$$

     \item For each $v_C^a\in\Theta$ of Type 2, $\tau(v_C^a)\in\Who$ is the vertex type whose unique non-trivial based partition is 
     $$\tau(v_C^a)_a=\{\{a\},C,(\G-\st(a))-C\}.$$ 
 \end{itemize}


    Given a clique $\sigma=\{x_0,\ldots,x_k\}\in\mathcal{P}(\Theta)$, \cite[Proposition 5.5]{DayWade} implies that the corresponding partial conjugations commute pairwise, which by \cite[Lemma 5.4]{ArMar} implies that the vertex types $\tau(x_0),\ldots,\tau(x_k)$ defined above are pairwise compatible. As a consequence (\cite[Lemma 5.6]{ArMar}), there is a least upper bound $\tau(x_0)\vee\cdots\vee \tau(x_k)$. We set 
    $$\phi(\sigma)=\tau(x_0)\vee\cdots\vee \tau(x_k),$$ so we have a map
    $$
\phi: \mathcal{P}(\Theta)\to \Wh_\G^0
    $$
    which is obviously a poset map.

    \begin{lem}\label{lem:CD} Let $\mathcal{H}=\{\Stab(\tau)\ |\ \tau\in\Who\}$ be the set of stabilizers of all the possible vertex types in $\Who$. There is a commutative diagram
   \[
  \begin{tikzcd}
     \mathcal{P}(\Theta) \arrow{r}{\phi} \arrow[swap]{dr}{\alpha} & \Who
     \arrow{d}{\pi} \\
     & \mathcal{H}
  \end{tikzcd}
\]
where $\pi(\tau)=\Stab(\tau)$ and $\alpha(\sigma)$ is the special subgroup of $\POut(A_\G)=A_\Theta$ generated by $\sigma$.
    \end{lem}
    \begin{proof}  We have to check that for any clique $\sigma=\{x_0,...,x_k\}\in\mathcal{P}_0(\Theta)$, the subgroup 
    \[\Stab(\phi(\sigma))=\Stab(\tau(x_0)\vee\ldots\vee\tau(x_k))\]
    is the special subgroup generated by $\sigma$. We argue by induction on $k$. 
    
    If $k=0$, the result is obvious by definition. The inductive step would also be obvious if we could prove that the group $\Stab(\tau\vee\tau')$ is generated by $\Stab(\tau)$ and $\Stab(\tau')$ for each pair $\tau,\tau'$ of compatible vertex types. This is not true in general, however, it is true in some cases. We divide the proof in several steps.
    
\smallskip
    \noindent\underline{Step 1}: Let $\tau,\tau'$ be compatible vertex types. Then $\Stab(\tau\vee\tau')=\langle\Stab(\tau),\Stab(\tau')\rangle$ if
    \begin{itemize}
        \item[(i)] either the sets $\{a\in\G\mid\tau_a\text{ is not trivial}\}$ and $\{a\in\G\mid\tau'_a\text{ is not trivial}\}$
        are disjoint,
        
\item[(ii)] or $\tau,\tau'$ have only one non trivial based partition at, say $a\in G$, and these based partitions are of the form
\[\begin{aligned}
    \tau_a&=\{\{\underline{a}\},L_1,L_2\},\\
    \tau'_a&=\{\{\underline{a}\},P_1,\ldots,P_t\}\\
\end{aligned}\]
with $L_1\subseteq P_1$ (thus $P_2,\ldots,P_t\subseteq L_0$).
        \end{itemize}
Item (i) of Step 1 follows from the fact that by definition, $\Stab(\tau\vee\tau')$ is generated by the groups $\Stab((\tau\vee\tau')_a)$ for $a\in\G$ and (i)  implies that the based partitions of $\tau\vee\tau'$ coincide with those of either $\tau$ or $\tau'$. 
For (ii), note that the only non trivial based partition of $\tau\vee\tau'$ is
$$(\tau\vee\tau')_a=\{\{\underline{a},L_1,P_1-L_1,\ldots,P_t\}\}.$$
As $C^a_{P_1-L_1}=C^a_{P_1}(C^a_{L_1})^{-1}$, we see that $C^a_{P_1-L_1}$ lies in the subgroup generated by $\Stab(\tau)$ and $\Stab(\tau')$, so the claim follows. 

 \smallskip
    \noindent\underline{Step 2}: We may assume that  for all the vertex types in the family $\tau(x_0),\ldots,\tau(x_k)$ the only non trivial based partition is based at the same $a\in\G$. To see it, note that in other case we could split $\sigma=\sigma_1\sqcup\sigma_2$ so that the based partitions of $\phi(\sigma_1)$ and $\phi(\sigma_2)$ are as in (i) of Step 1 so we would have
    \[\Stab(\phi(\sigma))=\langle\Stab(\phi(\sigma_1)),\Stab(\phi(\sigma_2))\rangle\]
    and by the induction hypothesis we would get the result.
    
    Observe moreover that by construction, the non trivial based partitions in the vertex types $\tau(x_0),\ldots,\tau(x_k)$ are all of the form $\tau(x_i)_a=\{\{\underline{a},L_1^i,L_2^i\}\}$ where with the notation of definition \ref{def:DayWadeRAAG}, $L_1^i$ is either  what we have called $L(e)$ or one of the trees of the forest $\Delta_a$. 

\smallskip
    \noindent\underline{Step 3}: With the same notation as in Step 2, we may reorder the elements $x_0,...,x_k$ and the petals of the based partitions $\tau(x_0)_a,\ldots,\tau(x_k)_a$ so that $L_1^0\subseteq L_1^1,L_1^2,\ldots,L_1^k$. To show this we consider several possible cases.
\begin{itemize}
    \item[\bf{3.1.}] Assume first that all of $x_0,\ldots,x_k$ are of type  $v^a_{L(e_i)}$ and all the edges $e_0,\ldots,e_k$ lie in the same maximal tree $T\subseteq\Delta_a$. The fact that the base point of $T$ is a leave implies that the family of subtrees $L(e_i)$ is partially ordered by containment, in fact for each $i,j$ we have either $L(e_i)\subseteq L(e_j)$, or $L(e_j)\subseteq L(e_i)$, or they are disjoint. Therefore we may reorder $x_0,\ldots,x_k$ so that $L(e_0)$ is minimal. This implies that for the others $L(e_i)$, $i\neq 0$, we have either $L(e_0)\subseteq L(e_i)$ or $L(e_0)\subseteq\G-\st(a)-L(e_i)$ and the claim follows upon a possible reorder of the petals in each $\tau(x_i)_a$ for $i=1,\ldots,k$. 

\item[\bf{3.2.}] If all of $x_0,\ldots,x_k$  are of type 2 we may choose any ordering of the elements $x_0,...,x_k$ and set
\[\tau(x_0)_a=\{a,T_0,(\G-\st(a))-T_0\},\]
\[\tau(x_i)_a=\{a,(\G-\st(a))-T_i,T_i\},\text{ for }i=2,\ldots,k\]
where $T_0,\ldots,T_k$ are trees in the forest $\Delta_a$ different from the preferred tree.

\item[\bf{3.3.}] Finally, in the remaining cases, we only have to choose a maximal tree $T\subseteq\Delta_a$ such that there are some $x_i$'s of type 1 of the form $v^a_{L(e)}$ for $e\in T$, order the set $x_0,...,x_k$ so that those comes first and ordered as in 3.1 and then order the petals of the remaining $x_i$'s so that the petal containing $T$ comes first. This shows the claim.
\end{itemize}

\smallskip
    \noindent\underline{Step 4}: End of the proof. We assume that $x_0,\ldots,x_k$ and the petals of each $\tau(x_i)_a$ are ordered as in Step 3. Let 
\[\tau'=\tau(x_1)\vee\ldots\vee\tau(x_k)=\phi(\{x_1,\ldots,x_k\}).\]
By induction, $\Stab(\tau')=\pi\phi(\{x_1,\ldots,x_k\})=\langle x_1,\ldots,x_k\rangle$. And, as $L_1^0\subseteq\cap_{i=1}^kL_1^i$, item (ii) in Step 1 implies
\[\Stab(\phi(\sigma))=\Stab(\tau(x_0)\vee\tau')=\langle\Stab(\tau(x_0)),\Stab(\tau')\rangle=\langle x_0,x_1,\ldots,x_k\rangle=\alpha(\sigma).\]

    \end{proof}

We can now show Theorem \ref{teo:homotopWhi-RAAG}.
    
\bigskip


\noindent{\sl Proof of Theorem \ref{teo:homotopWhi-RAAG}.} Consider the poset map
\[\phi:\mathcal{P}(\Theta)\to\Who\]
constructed above. We will show that 
 for any vertex type 
$\tau\in\Who$, the geometric realization of the set
\[\phi^{-1}(\tau_\leq)=\{\sigma\in\mathcal{P}_0(\Theta)\ |\ \tau\leq\phi(\sigma)\}\]
is contractible.  Using Quillen's poset Theorem (cf. Proposition 1.6, \cite{Q}) this will imply that $\phi$ induces a homotopy equivalence between the geometric realizations. 

To do that, we claim first that for any $\tau$, $\phi^{-1}(\tau_\leq)$ is intersection closed so whenever it is not empty, it has a smallest element and therefore its geometric realization $|\phi^{-1}(\tau_\leq)|$ is contractible. Let   $\sigma_1,\sigma_2\in\phi^{-1}(\tau_\leq)$ be cliques, then $\tau\leq\phi(\sigma_1),\phi(\sigma_2)$ so the commutativity of the diagram in Lemma \ref{lem:CD} implies
\[\Stab(\tau)\leq \Stab(\phi(\sigma_1))\cap\Stab(\phi(\sigma_2))=A_{\sigma_1}\cap A_{\sigma_2}.\]
Any intersection of special subgroups is also special so we have
\[A_{\sigma_1}\cap A_{\sigma_2}=A_{\sigma_1\cap\sigma_2}\]
and using again the commutativity of the diagram in Lemma \ref{lem:CD}, 
\[\Stab(\tau)\leq A_{\sigma_1\cap\sigma_2}=\Stab(\phi(\sigma_1\cap\sigma_2)).\]
Then, Remark \ref{rem:stab} implies
\[\tau\leq \phi(\sigma_1\cap\sigma_2)\]
thus $\sigma_1\cap\sigma_2\in\phi^{-1}(\tau_\leq)$.

This means that to prove the Theorem, it only remains to see that for any vertex type $\tau\in\Who$, the set $\phi^{-1}(\tau_\leq)$ is not empty.

We consider first the case of a vertex type with a unique non-trivial based partition $\tau_a$ for some $a\in\Gamma$. We will also use $\tau_a$ to denote the vertex type. Let $\sigma_a$ be the clique in $\Theta$ consisting of all the elements of the form either $v^a_e$ or $v^a_C$ in Definition \ref{def:DayWadeRAAG}. Then, $\phi(\sigma_a)=\nu_a$ where $\nu_a$ has only one non trivial based partition that we again denote as the vertex type and one easily checks that 
\[\nu_a=\{\{a\},P_1,\ldots,P_s\}\]
where the $P_i$ are all the components of $\G-\st(a)$. As $\nu_a$ is an upper bound for all the partitions based at $a$, this implies $\tau_a\leq\phi(\sigma_a)$. 
So we have $\sigma_a\in\phi^{-1}((\tau_a)_\leq)$ and we have seen before that this set has a smallest element that we call $\sigma(\tau_a)$, i.e., $\sigma(\tau_a)$ is the smallest clique $\sigma(\tau_a)\subseteq\Theta$ such that $\tau_a\leq\phi(\sigma(\tau_a))$. We have $\Stab(\tau_a)\leq\Stab(\phi(\sigma(\tau_a)))$ and the commutativity of the diagram in Lemma \ref{lem:CD} implies that  
\begin{equation}\label{eq:stabilizadores}\Stab(\tau_a)\leq\Stab(\phi(\sigma(\tau_a)))=\langle\sigma(\tau_a)\rangle=A_{\sigma(\tau_a)}\end{equation} so from the minimality of $\sigma(\tau_a)$ we deduce that $A_{\sigma(\tau_a)}$ is the smallest special spherical subgroup of the RAAG $A_\Theta=\POut(A_\G)$ that contains $\Stab(\tau_a)$.

For the general case, let $\tau$ be an arbitrary vertex type and $\tau_{a_1},\ldots,\tau_{a_k}$ its set of non trivial based partitions. We may identify each $\tau_{a_i}$ with the vertex type having $\tau_{a_i}$ as its only non trivial based partition and then
\[\tau=\tau_{a_1}\vee\cdots\vee\tau_{a_k}.\]
For each $a_i$, consider the clique $\sigma(\tau_{a_i})\subseteq\Theta$ so that $\tau_{a_i}\leq\phi(\sigma(\tau_{a_i}))$ and $\sigma(\tau_{a_i})$ is smallest possible. Then
$\Stab(\tau_{a_i})\leq A_{\sigma(\tau_{a_i})}$.

As the vertex types $\tau_{a_1},\ldots,\tau_{a_k}$ are pairwise compatible, the subgroups $\Stab(\tau_{a_1}),\ldots,\Stab(\tau_{a_k})$ of $A_\Theta=\POut(A_\G)$ commute pairwise so iterating Lemma \ref{lem:commuting} we deduce that the special subgroups \[A_{\sigma(\tau_{a_1})},\ldots,A_{\sigma(\tau_{a_k})}\] also do. This implies that the (disjoint) union
\[\sigma(\tau)=\sigma(\tau_{a_1})\cup\ldots\cup\sigma(\tau_{a_k})\] is also a simplex of $\Theta$ and by definition of $\phi$
\[\phi(\sigma(\tau))=\phi(\sigma(\tau_{a_1}))\vee\ldots\vee\phi(\sigma(\tau_{a_k})).\] 
Therefore
\[\tau=\tau_{a_1}\vee\cdots\vee\tau_{a_k}\leq\phi(\sigma(\tau_{a_1}))\vee\cdots\vee\phi(\sigma(\tau_{a_k}))=\phi(\sigma(\tau))\]
and we are done.\qed

\section{Simple Complexes of groups}\label{sec:simplecomplexes}

\begin{defi}
A \emph{small category without loops} (\emph{scwol}) $\mathcal{X}$ consists of a vertex set $V(\mathcal{X})$ and an edge set $E(\mathcal{X})$, together with three maps. Two of these maps, $i:E(\mathcal{X})\to V(\mathcal{X})$ and $t:E(\mathcal{X})\to V(\mathcal{X})$, assign to each $\alpha\in E(\mathcal{X})$ an \emph{initial vertex} $i(\alpha)$ and a \emph{terminal vertex} $t(\alpha)$. The third map is a composition $\circ: E^{(2)}(\mathcal{X})\to E(\mathcal{X})$,
where $E^{(2)}(\mathcal{X})$ denotes the set of pairs $(\alpha,\beta)\in E(\mathcal{X})\times E(\mathcal{X})$ such that $i(\alpha)=t(\beta)$. These sets and maps satisfy:
\begin{enumerate}
    \item[(i)] For all $(\alpha,\beta)\in E^{(2)}(\mathcal{X})$, we have $i(\alpha\circ\beta)=i(\beta),\quad t(\alpha\circ\beta)=t(\alpha)$.
    \item[(ii)] For all $\alpha,\beta,\gamma\in E(\mathcal{X})$ such that $i(\alpha)=t(\beta)$ and $i(\beta)=t(\gamma)$, we have $(\alpha\circ\beta)\circ\gamma=\alpha\circ(\beta\circ\gamma)$.
    \item[(iii)] For each $\alpha\in E(\mathcal{X})$, $i(\alpha)\neq t(\alpha)$.
\end{enumerate}
\end{defi}

The most important example of a scwol for our purposes comes from a poset:
\begin{ej}
Let $\mathcal{P}$ be a poset. It defines a scwol $\mathcal{X}$ as follows:
\begin{itemize}
    \item The vertex set is $V(\mathcal{X})=\mathcal{P}$.
    \item The edge set is $E(\mathcal{X})=\{(\sigma,\tau)\in \mathcal{P}\times \mathcal{P} \mid \tau<\sigma\}$.
    \item For an edge $(\sigma,\tau)$, we set $i(\sigma,\tau)=\tau$ and $t(\sigma,\tau)=\sigma$. 
    \item If $\tau<\sigma<\rho$, the composition is $(\rho,\sigma)\circ(\sigma,\tau)=(\rho,\tau)$.
\end{itemize}
\end{ej}

\begin{defi}
Let $\mathcal{X}$ be a scwol. A \emph{simple complex of groups} $\mathcal{G}(\mathcal{X})=(G_v,\psi_\alpha)$ over $\mathcal{X}$ consists of:
\begin{enumerate}
    \item[(i)] For each $v\in V(\mathcal{X})$, a group $G_v$, called the \emph{local group at $v$}.
    \item[(ii)] For each $\alpha\in E(\mathcal{X})$, an injective homomorphism 
    $$\psi_\alpha:G_{i(\alpha)}\to G_{t(\alpha)}$$
    such that whenever defined, $\psi_{\alpha\beta}=\psi_\alpha\circ\psi_\beta$.
\end{enumerate}
\end{defi}

\begin{defi}
Let $\mathcal{G}(\mathcal{X})$ be a simple complex of groups and $G$ a group. A morphism $\phi=(\phi_v,\phi(\alpha)):\mathcal{G}(\mathcal{X})\to G$ consists of:
\begin{enumerate}
    \item[(i)] A group homomorphism $\phi_v:G_v\to G$ for each $v\in V(\mathcal{X})$.
    \item[(ii)] An element $\phi(\alpha)\in G$ for each $\alpha\in E(\mathcal{X})$ such that
    $$\mathrm{Ad}(\phi(\alpha))\circ \phi_{i(\alpha)} = \phi_{t(\alpha)} \circ \psi_\alpha,\quad 
    \phi(\alpha\beta)=\phi(\alpha)\phi(\beta),$$
    where $\mathrm{Ad}(\phi(\alpha))$ denotes conjugation by $\phi(\alpha)$.
\end{enumerate}
\end{defi}

\begin{defi}
A complex of groups $\mathcal{G}(\mathcal{X})$ over a scwol $\mathcal{X}$ is called \emph{developable} if there exists a group $G$ and a morphism $\phi:\mathcal{G}(\mathcal{X})\to G$ that is injective on all local groups.
\end{defi}

This definition is equivalent to the original one (\cite[Corollary III.$\mathcal{C}$.2.15]{Bridson-Hafliger}), and it is usually easier to verify in examples.

\medskip

For every scwol $\mathcal{X}$, one can consider its geometric realization $\lvert \mathcal{X} \rvert$. If $\mathcal{X}$ has no multiple edges, this construction coincides with the geometric realization of a simplicial complex. In particular, if $\mathcal{X}$ is induced by a poset, then its geometric realization agrees with the usual geometric realization of the poset. The general construction can be found in \cite[Chapter III.C.1]{Bridson-Hafliger}.

For the next definition we will assume that the scwol $\mathcal{X}$ is connected (i.e., there is only one equivalence class in $V(\mathcal{X})$ under the relation generated by $i(\alpha)\sim t(\alpha)$ for $\alpha\in E(\mathcal{X})$). 

\begin{defi}
Let $\mathcal{G}(\mathcal{X})$ be a simple complex of groups over a connected scwol $\mathcal{X}$. Assume that each $G_v=\langle S_v\mid R_v\rangle$ is finitely presented. Choose a maximal tree $T$ in the 1-skeleton $\lvert\mathcal{X}\rvert^{(1)}$ of the geometric realization of $\mathcal{X}$. Let $E(\mathcal{X})^\pm = \{\alpha^+,\alpha^- \mid \alpha \in E(\mathcal{X})\}$. Then the fundamental group $\pi_1(\mathcal{G}(\mathcal{X}),T)$ is generated by
$$\left(\bigsqcup_{v\in V(\mathcal{X})} S_v\right) \sqcup E(\mathcal{X})^\pm,$$
subject to the relations:
\begin{enumerate}
    \item[(i)] $R_v$ for all $v\in V(\mathcal{X})$.
    \item[(ii)] $(\alpha^+)^{-1}=\alpha^-$ for all $\alpha\in E(\mathcal{X})$.
    \item[(iii)] $(\alpha\beta)^+=\alpha^+\beta^+$ for all $(\alpha,\beta)\in E^{(2)}(\mathcal{X})$.
    \item[(iv)] $\psi_\alpha(s) = \alpha^+ s \alpha^-$ for all $\alpha\in E(\mathcal{X}), s\in S_{i(\alpha)}$.
    \item[(v)] $\alpha^+=1$ for all $\alpha\in E(T)$.
\end{enumerate}
\end{defi}

The fundamental group does not depend on the choice of $T$ (\cite[Theorem III.$\mathcal{C}$.3.7]{Bridson-Hafliger}), so we denote it simply as $\pi_1(\mathcal{G}(\mathcal{X}))$.

\medskip

Assume $\mathcal{G}(\mathcal{X})$ is developable. Then there is an action of $G=\pi_1(\mathcal{G}(\mathcal{X}))$ on the space
$$
\widetilde{\mathcal{X}} = (G \times \lvert\mathcal{X}\rvert)/\sim,
$$
where $(g,x)\sim(g',x')$ if and only if $x=x'$ and $gG_{\sigma(x)} = g'G_{\sigma(x)}$. Here $\sigma(x)$ denotes the unique minimal cell of $|\mathcal{X}|$ containing $x$. We denote by $[g,x]$ the equivalence class of $(g,x)$ and then the action of $G$ is
$$
h\cdot [g,x] = [hg,x] \quad \text{for all } h,g\in G,\, x\in \lvert\mathcal{X}\rvert.
$$
Cell stabilizers are of the form $gG_v g^{-1}$ for $g\in G$ and $v\in V(\mathcal{X})$. A strong fundamental domain is
$$
W = \{ [1,x] \mid x \in \lvert \mathcal{X}\rvert \},
$$
which is isomorphic to $\lvert \mathcal{X} \rvert$. Observe that $\mathcal{X} = \widetilde{\mathcal{X}}/G$, so the action is cocompact if $\mathcal{X}$ is finite, and it is obviously admissible. 

\begin{teo}\label{theorem:SimpleComplexGroups}
Assume that $\mathcal{X}$ is a finite scwol and that $\mathcal{G}(\mathcal{X})$ is a developable simple complex of groups with each $G_v$ of type $\mathrm{FP}_\infty$. Let $G = \pi_1(\mathcal{G}(\mathcal{X}))$, and assume that both $\lvert \mathcal{X} \rvert$ and $\lvert \widetilde{\mathcal{X}} \rvert$ are contractible. 
 Then, Hypothesis (*) holds for the action of $G$ on $\lvert \widetilde{\mathcal{X}} \rvert$ with fundamental domain  $\lvert \mathcal{X} \rvert$. 
\end{teo}

This means that if the groups $G_v$, $v \in V(\mathcal{X})$ satisfy the conditions in Lemma \ref{lem:densitycond}, then we can apply 
Theorem~\ref{teo:A}. We will next consider some families of groups for which this is the case.


\subsection{Artin groups satisfying the \texorpdfstring{$K(\pi,1)$}-conjecture}

\begin{defi}
Let $\Gamma$ be a finite simplicial graph equipped with a labeling map $l:E(\Gamma)\to\{2,3,4,\dots\}$.  The \emph{Artin group} associated to $\Gamma$ is
$$A_\Gamma = \left\langle v \in V(\Gamma) \,\middle|\, 
\underbrace{uvuv\cdots}_{l(e)\ \text{letters}} =
\underbrace{vuvu\cdots}_{l(e)\ \text{letters}},\; \text{for all } e=\{u,v\}\in E(\Gamma)
\right\rangle.$$
\end{defi}


Given an Artin group $A_\Gamma$ and an induced subgraph $\Delta \subset \Gamma$, the subgroup of $A_\Gamma$ generated by $\{v\}_{v\in V(\Delta)}$ is naturally isomorphic to $A_\Delta$ (cf. \cite[Chapter II, Lemma 4.11]{VanDerLek}). These subgroups are called \emph{special subgroups}.  For every Artin group $A_\Gamma$, the associated Coxeter group is $W_\Gamma = A_\Gamma \big/ \langle\!\langle v^2 \mid v\in V(\Gamma) \rangle\!\rangle$. We say that $A_\Gamma$ is \emph{spherical} if $W_\Gamma$ is finite. Since finite Coxeter groups can be classified in terms of the graph (c.f. \cite{Coxeter}), this condition can be checked directly from $\Gamma$.

\medskip

Fix an Artin group $A_\Gamma$ and consider the \emph{spherical poset}
$$
\mathcal{P} = \{ A_\Delta \mid \Delta \subset \Gamma \text{ is spherical} \},
$$
with the natural order given by inclusions $A_\Delta \leq A_\Omega$ whenever $\Delta \subset \Omega$. Define the \emph{coset poset}
$$
C\mathcal{P} = \{ g A_\Delta \mid g \in A_\Gamma, A_\Delta \in \mathcal{P} \}.
$$
The geometric realization $C(\mathcal{P})$ of this poset  is called the \emph{Deligne complex}. 
It is conjectured that $C(\mathcal{P})$ is contractible; this is equivalent to the well-known $K(\pi,1)$-conjecture for Artin groups. For a detailed introduction and more info about this conjecture see \cite{Paris}.

\medskip

Let $\mathcal{X}$ be the scwol induced by $\mathcal{P}$. Define a simple complex of groups $\mathcal{G}(\mathcal{X})$ with local groups $A_\Delta \in \mathcal{P}$ and morphisms given by inclusions $A_\Delta \hookrightarrow A_\Omega$ for $\Delta \subset \Omega$. Consider the morphism $\phi:\mathcal{G}(\mathcal{X}) \to A_\Gamma$ given by $\phi_\Delta:A_\Delta \hookrightarrow A_\Gamma$ for each $A_\Delta \in \mathcal{P}$ and $\phi(\alpha) = 1$ for each $\alpha \in E(\mathcal{X})$. Since this is injective on the local groups, $\mathcal{G}(\mathcal{X})$ is developable.

\begin{lem}\label{Artin1} We have an isomorphism $\pi_1(\mathcal{G}(\mathcal{X})) \cong A_\Gamma$ and an equivariant homotopy equivalence $\widetilde{\mathcal{X}} \simeq C(\mathcal{P})$.
\end{lem}

\begin{proof}
Let $T$ be the maximal tree of $\abs{\mathcal{X}}^{(1)}$ consisting of edges $(A_\Delta,A_\emptyset)$ for each $A_\Delta \in \mathcal{P}$. For an edge $\alpha=(A_\Omega,A_\Delta)$, using relations (iii) and (v) we have
$$
(A_\Omega,A_\Delta)^+ = (A_\Omega,A_\Delta)^+ (A_\Delta,A_\emptyset)^+ = \big( (A_\Omega,A_\Delta)(A_\Delta,A_\emptyset) \big)^+ = (A_\Omega,A_\emptyset)^+ = 1.
$$
Hence, all edges in $E(\mathcal{X})^\pm$ become trivial in the presentation. By relation (iv), for each $\Delta \subset \Omega$ we identify $A_\Delta$ with its image in $A_\Omega$, so the presentation of $\pi_1(\mathcal{G}(\mathcal{X}),T)$ coincides with that of $A_\Gamma$.

The second claim follows from the fact that $\abs{\mathcal{X}}=\abs{\mathcal{P}}$. More explicitely,
each $x \in \abs{\mathcal{X}}$ lies in an open simplex corresponding to a chain of spherical subgraphs $\Delta_1 \subset \dots \subset \Delta_k$. Define
$\Phi:\widetilde{\mathcal{X}} \to C(\mathcal{P})$ by $\Phi([g,x]) = gA_{\Delta_1} \subset \dots \subset gA_{\Delta_k}$. This map yields a homotopy equivalence.
\end{proof}


\begin{teo}\label{theorem:SigmaArtin}
Let $A_\Gamma$ be an Artin group satisfying the $K(\pi,1)$-conjecture. Then the following are equivalent:
\begin{itemize}
    \item[(i)] $\Sigma^n(A_\Gamma,\mathbb{Z})$ is dense in $S(A_\Gamma)$,
    \item[(ii)] $\Sigma^n(A_\Gamma,\mathbb{Z}) \neq \emptyset$,
    \item[(iii)] $\widehat{\Gamma}$ is $(n-1)$-acyclic,
\end{itemize}
where $\widehat{\Gamma}$ denotes the \emph{spherical flag complex} obtained by adding an $(n-1)$-simplex for each spherical subgraph of $\Gamma$ with $n$ vertices.
\end{teo}

\begin{proof}
Since $\abs{\mathcal{X}}$ has a minimal element, it is contractible. Moreover, $C(\mathcal{P})$ is contractible because we are assuming that $A_\Gamma$ satisfies the $K(\pi,1)$-conjecture. Every special subgroup of an Artin group satsfying the $K(\pi,1)$-conjecture also satisfies the $K(\pi,1)$-conjecture (c.f. \cite{Godelle-Paris}), so the stabilizers are all $\mathrm{FP}_\infty$.
Thus, Lemma \ref{Artin1} implies that the hypotheses of Theorem \ref{theorem:SimpleComplexGroups} are satisfied.

By \cite[Section 3]{Escartin-Martinez}, the image of $Z(A_\Delta)$ in the abelianization is infinite. Hence, using Lemma \ref{lem:densitycond} (ii) we see that also the hypotheses of Theorem \ref{teo:A} hold. Finally, recalling that $W_{\Fsing} = \{ A_\Delta \mid \emptyset \neq \Delta \subset \Gamma \text{ spherical} \}$ is isomorphic to $\widehat{\Gamma}$, completes the proof.
\end{proof}
\subsection{Graph products of groups}

\begin{defi}
Let $\Gamma$ be a finite simplicial graph and assume that the vertex set of $\Gamma$ is labeled by a family of groups $\{G_v\}_{v\in V(\Gamma)}$.  
The associated \emph{graph product} is defined by the presentation
$$\mathcal{G}(\Gamma,\{G_v\}_{v\in V(\Gamma)}) :=
\Big\langle\, G_v \text{ for all } v\in V(\Gamma) \;\big|\;
[G_u,G_v]=1 \text{ for all } \{u,v\}\in E(\Gamma) \,\Big\rangle.$$
\end{defi}

To simplify notation, we will denote graph products as $\mathcal{G}_\Gamma$ once the vertex groups are fixed.
Given a graph product $\mathcal{G}_\Gamma$ and an induced subgraph $\Delta\subset\Gamma$, the normal form theorem for graph products (cf.\ \cite{Green}) implies that the subgroup of $\mathcal{G}_\Gamma$ generated by $\{G_v\}_{v\in V(\Delta)}$
 is naturally isomorphic to $\mathcal{G}_\Delta$.  
These subgroups are called \emph{special subgroups}.

\medskip

In \cite{Davis}, a CAT(0) cube complex was constructed for arbitrary graph products of groups; this complex is now known as the \emph{Davis complex} associated to $\mathcal{G}_\Gamma$.  
It is constructed analogously to the Deligne complex for Artin groups.  
Consider the poset of all special subgroups associated to complete subgraphs:
$$\mathcal{P} \;=\; \{\, \mathcal{G}_\Delta \mid \Delta \subset \Gamma \text{ is complete} \,\},$$
ordered by inclusion: $\mathcal{G}_\Delta \le \mathcal{G}_\Omega$ whenever $\Delta \subset \Omega$.  
The Davis complex associated to $\mathcal{G}_\Gamma$ is defined as $C(\mathcal{P})$, the geometric realization of the coset poset $\mathcal{G}_\Gamma \mathcal{P}$.

\medskip

Let $\mathcal{X}$ be the scwol induced by $\mathcal{P}$.  
Define a simple complex of groups $\mathcal{G}(\mathcal{X})$ with local groups $\mathcal{G}_\Delta \in \mathcal{P}$ and morphisms given by the inclusions $\mathcal{G}_\Delta \hookrightarrow \mathcal{G}_\Omega$ for $\Delta \subset \Omega$.  
Exactly as in the Artin group case, one proves that $\mathcal{G}(\mathcal{X})$ is developable, that $\pi_1(\mathcal{G}(\mathcal{X})) \cong \mathcal{G}_\Gamma,$ and that $\widetilde{\mathcal{X}} \cong C(\mathcal{P})$.

\begin{teo}\label{theorem:SigmaGraphProd}
Let $\mathcal{G}_\Gamma$ be a graph product such that each $G_v$ is of type $\mathrm{FP}_\infty$ and
$\Sigma^n(G_v,\mathbb{Z}) \subseteq S(G_v)$ is dense for every $v \in V(\Gamma)$.  
Then the following are equivalent:
\begin{itemize}
\item[(i)] $\Sigma^n(\mathcal{G}_\Gamma,\mathbb{Z})$ is dense in $S(\mathcal{G}_\Gamma)$;
\item[(ii)] $\Sigma^n(\mathcal{G}_\Gamma,\mathbb{Z}) \cap -\Sigma^n(\mathcal{G}_\Gamma,\mathbb{Z})$ is dense in $S(\mathcal{G}_\Gamma)$;
\item[(iii)] $\Sigma^n(\mathcal{G}_\Gamma,\mathbb{Z}) \cap -\Sigma^n(\mathcal{G}_\Gamma,\mathbb{Z}) \neq \emptyset$;
\item[(iv)] $\widehat{\Gamma}$ is $(n-1)$-acyclic.
\end{itemize}
Here $\widehat{\Gamma}$ denotes the flag complex of $\G$.
\end{teo}

\begin{proof}
All stabilizers are conjugates of $\mathcal{G}_\Delta = \prod_{v\in V(\Delta)} G_v$, for $\emptyset\neq\Delta \subseteq \Gamma$ complete. Thus all stabilizers are of type $\mathrm{FP}_\infty$, and the hypotheses of Theorem \ref{theorem:SimpleComplexGroups} are satisfied. 

By the direct product formula for the Sigma invariants (cf.\ \cite{Bieri-Geoghegan}) we have $\Sigma^n(\mathcal{G}_\Delta,\mathbb{Z}) \subseteq S(\mathcal{G}_\Delta)$ is dense for each complete $\Delta\subset\Gamma$. Moreover, stabilizers are retracts: to see it observe that the map $\mathcal{G}_\Gamma\to\mathcal{G}_\Delta$ sending any $x\in G_v$ to itself if $v\in\Delta$ and to the identity in other case is well defined and restricts to the identity on $\mathcal{G}_\Delta$. And as a consequence, the map 
\[(\mathcal{G}_\Delta)_\mathrm{ab}\to(\mathcal{G}_\Gamma)_\mathrm{ab}\]
is injective. This means that Lemma \ref{lem:densitycond} (iii) applies so also the hypotheses of Theorem \ref{teo:A} hold.
To finish the proof, observe that $W_{\Fsing}$ is the subcomplex of $C(\mathcal{P})$ spanned by $\{\;\mathcal{G}_\Delta \mid \emptyset \neq \Delta \subset \Gamma, \mathcal{G}_\Delta \text{ infinite}\;\}$, and this subcomplex is isomorphic to $\widehat{\Gamma}$. 
\end{proof}

\end{document}